\documentclass[12pt]{amsart}

\usepackage{setspace}
\usepackage{amsthm}
\usepackage{amssymb}
\usepackage{amsmath}
\usepackage{eucal}
\usepackage{graphicx}
\usepackage{amsfonts}
\numberwithin{equation}{section}
\newcommand{\C}{\mathbb{C}}
\newcommand{\R}{\mathbb{R}}
\newcommand{\x}{\mathbf{x}}
\newcommand{\y}{\mathbf{y}}
\newcommand{\z}{\mathbf{z}}
\newcommand{\p}{\mathbf{p}}
\renewcommand{\Re}{\textrm{Re}}

\newcommand{\bea}{\begin{eqnarray}}
\newcommand{\eea}{\end{eqnarray}}
\newcommand{\be}{\begin{eqnarray*}}
\newcommand{\ee}{\end{eqnarray*}}
\newtheorem{theorem}{Theorem}[section]
\newtheorem{lemma}[theorem]{Lemma}
\newtheorem{corollary}[theorem]{Corollary}
\newtheorem{definition}[theorem]{Definition}

\allowdisplaybreaks[1]

\begin{document}

\title[Newton's Method Backpropagation]{Newton's Method Backpropagation for Complex-Valued Holomorphic Multilayer Perceptrons}
\author[Diana Thomson La Corte and Yi Ming Zou]{Diana Thomson La Corte and Yi Ming Zou}

\maketitle

\begin{abstract}
The study of Newton's method in complex-valued neural networks faces many difficulties.  In this paper, we derive Newton's method backpropagation algorithms for complex-valued holomorphic multilayer perceptrons, and investigate the convergence of the one-step Newton steplength algorithm for the minimization of real-valued complex functions via Newton's method.  To provide experimental support for the use of holomorphic activation functions, we perform a comparison of using sigmoidal functions versus their Taylor polynomial approximations as activation functions by using the algorithms developed in this paper and the known gradient descent backpropagation algorithm.   Our experiments indicate that the Newton's method based algorithms, combined with the use of polynomial activation functions, provide significant improvement in the number of training iterations required over the existing algorithms.
\end{abstract}


\section{Introduction}

\par
The use of fully complex-valued neural networks to solve real-valued as well as complex-valued problems in physical applications has become increasingly popular in the neural network community in recent years \cite{Hirose2012,Hirose2011,Mukherjee2012}.
Complex-valued neural networks pose unique problems, however. Consider the problem of choosing the activation functions for a neural network.  Real-valued activation functions for real-valued neural networks are commonly taken to be everywhere differentiable and bounded.  Typical activation functions used for real-valued neural networks are the sigmoidal, hyperbolic tangent, and hyperbolic secant functions
\begin{equation*}
 f(x)=\frac{1}{1+\exp(-x)}, \textrm{ } \tanh(x)=\frac{e^x-e^{-x}}{e^x+e^{-x}}, \textrm{ and } \textrm{sech}(x)=\frac{2}{e^x+e^{-x}}.
\end{equation*}
For activation functions of complex-valued networks, an obvious choice is to use the complex counterparts of these real-valued functions.  However, as complex-valued functions, these functions are no longer differentiable and bounded near $0$, since they have poles near $0$.  Different approaches have been proposed in the literature to address this problem.

\par
Liouville's theorem tells us that there is no non-constant complex-valued function which is both bounded and differentiable on the whole complex plane \cite{Conway1978}.  On the basis of Liouville's theorem, \cite{Georgiou1992} asserts that an entire function is not suitable as an activation function for a complex-valued neural network and claims boundedness as an essential property of the activation function.  Some authors followed this same reasoning and use the so-called ``split'' functions of the type $f(z)=f(x+iy)=f_1(x)+if_2(y)$ where $f_1,f_2$ are real-valued functions, typically taken to be one of the sigmoidal functions \cite{Jalab2011,Kim2001,Pande2007}.  Such activation functions have the advantage of easily modeling data with symmetry about the real and imaginary axes. However, this yields complex-valued neural networks which are close to real-valued networks of double dimensions and are not fully complex-valued \cite{Hirose2012}.  Amplitude-phase-type activation functions have the type $f(z)=f_3(\vert z\vert) \exp (i\text{arg}(z))$ where $f_3$ is a real-valued function.  These process wave information well, but have the disadvantage of preserving phase data, making the training of a network more difficult \cite{Hirose2012,Kim2001}.  Some authors forgo complex-valued activation functions entirely, choosing instead to scale the complex inputs using bounded real-valued functions which are differentiable with respect to the real and imaginary parts \cite{Amin2009-2,Amin2009,Amin2008}.  While this approach allows for more natural grouping of data for classification problems, it requires a modified backpropagation algorithm to train the network, and again the networks are not fully complex-valued.  Other authors choose differentiability over boundedness and use elementary transcendental functions \cite{Hanna2003,Kim2001,Kim2002}.  Such functions have been used in complex-valued multilayer perceptrons trained using the traditional gradient descent backpropagation algorithm and in other applications \cite{Burse2011,Li2005,Savitha2011}.  However, the problem of the existence of poles in a bounded region near $0$ presents again.  Though one can try to scale the data to avoid the regions which contain poles \cite{Leung1991}, this does not solve the problem, since for unknown composite functions, the locations of poles are not known {\it a priori}.  The exponential function $\exp(z)$ has been proposed as an alternative to the elementary transcendental functions for some complex-valued neural networks, and experimental evidence suggests better performance of the entire exponential function as activation function than those with poles \cite{Savitha2009}.

\par
In this paper, we will derive the backpropagation algorithm for fully complex-valued neural networks based on Newton's method.  We compare the performances of using the complex-valued sigmoidal activation function and its Taylor polynomial approximations.  Our results give strong supporting evidence for the use of holomorphic functions, in particular polynomial functions, as activation functions for complex-valued neural networks.  Polynomials have been used in fully complex-valued functional link networks \cite{Amin2011,Amin2012}, however their use is limited as activation functions for fully complex-valued multilayer perceptrons.  Polynomial functions are differentiable on the entire complex plane and are underlying our computations due to Taylor's Theorem, and they are bounded over any bounded region.  Moreover, the complex Stone-Weierstrass Theorem implies that any continuous complex-valued function on a compact subset of the complex plane can be approximated by a polynomial \cite{Reed1980}.  Due to the nature of the problems associated with the activation functions in complex-valued neural networks, different choices of activation functions can only suit different types of neural networks properly, and one should only expect an approach to be better than the others in certain applications.

\par
We will allow a more general class of complex-valued functions for activation functions, namely the holomorphic functions. There are two important reasons for this.  The first one is that holomorphic functions encompass a general class of functions that are commonly used as activation functions.  They allow a wide variety of choices both for activation functions and training methods.  The second is that the differentiability of holomorphic functions leads to much simpler formulas in the backpropagation algorithms.  For application purpose, we will also consider the backpropagation algorithm using the pseudo-Newton's method, since it has computational advantage.  Our main results are given by Theorem \ref{thm:complexnewtbackprop}, Corollary \ref{cor:complexp-newtbackprop}, and Theorem \ref{thm:ConvergenceNewton}.  Theorem \ref{thm:complexnewtbackprop} gives a recursive algorithm to compute the entries of the Hessian matrices in the application of Newton's method to the backpropagation algorithm for complex-valued holomorphic multilayer perceptrons, and Corollary \ref{cor:complexp-newtbackprop} gives a recursive algorithm for the application of the pseudo-Newton's method to the backpropagation algorithm based on Theorem \ref{thm:complexnewtbackprop}.  The recursive algorithms we developed are analogous to the known gradient descent backpropagation algorithm as stated in Section III, hence can be readily implemented in real-world applications.  A problem with Newton's method is the choice of steplengths to ensure the algorithm actually converges in applications.  Our setting enables us to perform a rigorous analysis for the one-step Newton steplength algorithm for the minimization of real-valued complex functions using Newton's method.  This is done in Section VI.  Our experiments, reported in Section VII, show that the algorithms we developed use significantly fewer iterations to achieve the same results as the gradient descent algorithm.  We believe that the Newton's method backpropagation algorithm provides a valuable tool for fast learning for complex-valued neural networks as a practical alternative to the gradient descent methods.

\par
The rest of the paper is as follows.  In Section II, we define holomorphic multilayer perceptrons and set up our notation for the network architecture we use throughout the rest of the paper.  In Section III, we give a reformulation of the gradient descent backpropagation algorithm based on our setting of holomorphic neural networks.  In Section IV, we derive the backpropagation algorithm for holomorphic multilayer perceptrons using Newton's method, and in Section V we restirct the results of Section IV to the pseudo-Newton's method.  In Section VI we state the one-step Newton steplength algorithm, and in Section VII we report our experiments.  The appendices provide the detailed computations omitted from Section IV and a detailed proof, omitted from Section VI, of the convergence of the one-step Newton steplength algorithm for the minimization of real-valued complex functions.

\section{Holomorphic MLPs: Definition and Network Architecture}

\par
A well-used type of artificial neural network is the multilayer perceptron (MLP).  An MLP is built of several layers of single neurons hooked together by a network of weight vectors.  Usually the activation function is taken to be the same among a single layer of the network; the defining characteristic of the MLP is that in at least one layer, the activation function must be nonlinear.  If there is no nonlinear activation function, the network can be collapsed to a two-layer network \cite{Buchholz2008}.

\begin{definition}
A {\bf holomorphic MLP} is a complex-valued MLP in which the activation function in the layer indexed by $p$ of the network is holomorphic on some domain $\Omega_p\subseteq\C$.
\label{def:HolomorphicMLP}
\end{definition}

\par
Most of the publications on complex-valued neural networks with holomorphic activation functions deal with functions that have poles.  We will mainly focus on entire functions for the purpose of applying Newton's method.  For these functions, we do not have to worry about the entries of a Hessian matrix hitting the poles.  However, we will allow some flexibility in our setting and set up our notation for a general $L$-layer holomorphic MLP as follows (see Figure \ref{fig:networkdiagram}).

\begin{figure}[b]\begin{center}
\includegraphics[scale=0.6]{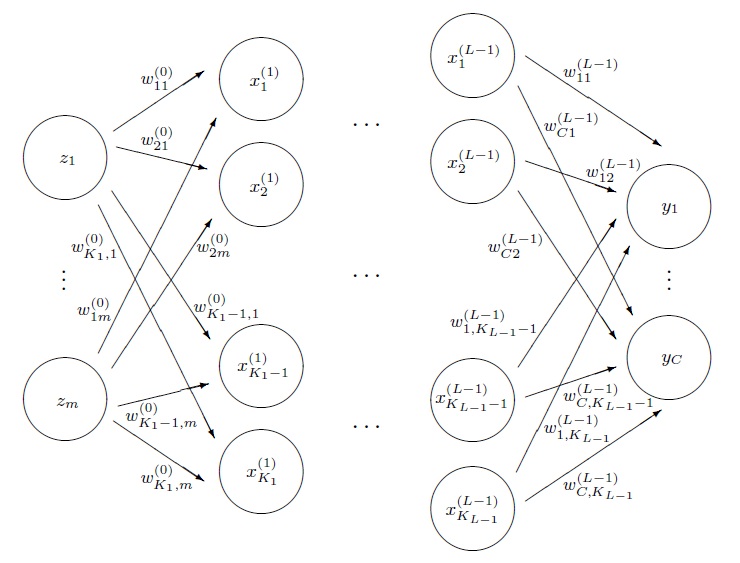}
\caption{Network Architecture}
\label{fig:networkdiagram}\end{center}
\end{figure}

\begin{itemize}
\item The input layer has $m=K_0$ input nodes denoted $$z_1=x^{(0)}_1,...,z_m=x^{(0)}_m.$$ \item There are $L-1$ hidden layers of neurons, and the $p$th ($1\le p\le L-1$) hidden layer contains $K_p$ nodes.  We denote the output of node $j$ ($j=1,...,K_p$) in the $p$th layer by $x^{(p)}_j.$  The inputted weights to the $p$th layer are denoted $w^{(p-1)}_{ji}$ ($j=1,...,K_p$, $i=1,...,K_{p-1}$), where $j$ denotes the target node of the weight in the $p$th layer and $i$ denotes the source node in the $(p-1)$th layer.  With these conventions we define the weighted net sum and the output of node $j$ of the $p$th layer by
\begin{equation*}
\left( x^{(p)}_j\right)^{\textrm{net}} = \sum_{i=1}^{K_{p-1}} w^{(p-1)}_{ji} x^{(p-1)}_i \textrm{ and } x^{(p)}_j = g_p \left(\left( x^{(p)}_j\right)^{\textrm{net}}\right),
\label{eq:hidlayeroutput}
\end{equation*}
where $g_p$, which is assumed to be holomorphic on some domain $\Omega_p\subseteq \C$, is the activation function for all the neurons in the $p$th layer.
\item The output layer has $C=K_L$ output nodes denoted $y_1=x^{(L)}_1,...,y_C=x^{(L)}_C.$  We define the weighted net sum and the output of node $l$ ($l=1,...,C$) by
\begin{equation*}
y_l^{\textrm{net}} = \sum_{k=1}^{K_{L-1}} w^{(L-1)}_{lk} x^{(L-1)}_k \textrm{ and }y_l = g_L \left(y_l^{\textrm{net}}\right),
\label{eq:outlayeroutput}
\end{equation*}
where $g_L$, which is assumed to be holomorphic on some domain $\Omega_L\subseteq \C$, is the activation function for all the neurons in the output layer.
\end{itemize}

\par

To train the network, we use a training set with $N$ data points $\left\{(z_{t1},...,z_{tm},d_{t1},...,d_{tC}) \, \vert \, t=1,...,N \right\}$, where $(z_{t1},...,z_{tm})$ is the input vector corresponding to the desired output vector $(d_{t1},...,d_{tC})$. As the input vector $(z_{t1},...,z_{tm})$ of the $t$th training point is propagated throughout the network we update the subscripts of the network calculations with an additional $t$ subscript to signify that those values correspond to the $t$th training point.  For example, $( x^{(p)}_{tj})^{\textrm{net}},\textrm{ } x^{(p)}_{tj},\textrm{ } y_{tl}^{\textrm{net}}, \textrm{ and } y_{tl}.$  Finally, we train the network by minimizing the standard sum-of-squares error function
\begin{equation*}
E=\frac{1}{N}\sum_{t=1}^N \sum_{l=1}^C \vert y_{tl}-d_{tl}\vert^2 = \frac{1}{N}\sum_{t=1}^N \sum_{l=1}^C \left(y_{tl}-d_{tl}\right)\left( \overline{y_{tl}}-\overline{d_{tl}}\right).
\label{eq:error}
\end{equation*}

\section{The Gradient Descent Backpropagation Algorithm}

\par

Minimization of the error function can be achieved through the use of the backpropagation algorithm.  Backpropagation trains the network by updating the output layer weights first in each step (via an update rule from some numerical minimization algorithm), then using the updated output layer weights to update the first hidden layer weights, and so on, ``backpropagating'' the updates throughout the network until a desired level of accuracy is achieved (usually, this is when the error function drops below a pre-fixed value).  In the case of real-valued neural networks, minimization of the error function by Newton's method is generally thought to be too computationally ``expensive,'' and several different methods are commonly used to approximate the Hessian matrices instead of computing them directly: for example the conjugate gradient, truncated Newton, Gauss-Newton and Levenberg-Marquardt algorithms \cite{Al-Haik2003,Beigi1993,Hagan1994,Mukherjee2012,Yu2011}.  In contrast, for complex-valued neural networks, gradient descent methods, which are known to give stable (albeit slow) convergence, are commonly used due to their relatively simple formulations, and a number of such minimization algorithms exist \cite{Goh2005,Leung1991,Zimmermann2011}.

\par

We reformulate a backpropagation algorithm using gradient descent according to our setting of the neural networks defined in Section II for two reasons: the algorithm has a much simpler formulation compared with the known ones \cite{Buchholz2008,Leung1991} due to the activation functions being taken to be holomorphic, and we will use it for comparison purpose.  A similar formulation of the backpropagation algorithm to ours is presented in \cite{Li2008}.  The formulas of gradient descent for complex functions can be found in \cite{Kreutz-Delgado2009}. We use the following vector notation.  For $1\le p\le L$, we denote the weights that input into the $p$th layer of the network using a vector whose components correspond to the target nodes:
\begin{equation*}
\mathbf{w}^{(p-1)}:=\left( w^{(p-1)}_{11},...,w^{(p-1)}_{1K_{p-1}},...,w^{(p-1)}_{K_p1},...,w^{(p-1)}_{K_pK_{p-1}} \right)^T,
\end{equation*}
that is, the components of $\mathbf{w}^{(p-1)}$ are
\begin{equation}\mathbf{w}^{(p-1)}\left[ (j-1)\cdot K_{p-1}+i\right] =w^{(p-1)}_{ji},
\label{eq:weightcomponents}
\end{equation}
where $j=1,...,K_p$, $i=1,...,K_{p-1}$.  Using this notation the update steps for backpropagation look like
\begin{equation}
\mathbf{w}^{(p-1)}(n+1)=\mathbf{w}^{(p-1)}(n)+\mu(n) \Delta \mathbf{w}^{(p-1)},
\label{eq:updatestep}
\end{equation}
where $\mathbf{w}^{(p-1)}(n)$ denotes the weight value after the $n$th iteration of the training algorithm, and $\mu(n)$ denotes the learning rate or steplength which is allowed to vary with each iteration.  

\par
Using the gradient descent method, the update for the $(p-1)$th layer of a holomorphic complex-valued neural network is (\cite{Kreutz-Delgado2009}, p. 60)
\begin{equation*}\Delta \mathbf{w}^{(p-1)}=-\left(\frac{\partial E}{\partial \mathbf{w}^{(p-1)}} \right)^*.\end{equation*}
Suppose the activation function for the $p$th layer of the network, $p=1,...,L$, satisfies $$\overline{g(z)}=g(\overline{z}).$$  Coordinate-wise the partial derivatives $\frac{\partial E}{\partial w^{(L-1)}_{lk}}$, taken with respect to the output layer weights $w^{(L-1)}_{lk}$, $l=1,...,C$, $k=1,...,K_{L-1}$, are given by
\begin{equation*}
\begin{split}
\frac{\partial E}{\partial w^{(L-1)}_{lk}} &= \frac{\partial}{\partial w^{(L-1)}_{lk}} \left[ \frac{1}{N}\sum_{t=1}^N \sum_{h=1}^C (y_{th}-d_{th})(\overline{y_{th}}-\overline{y_{th}}) \right]\\
&=\frac{1}{N} \sum_{t=1}^N \left[ \frac{\partial y_{tl}}{\partial w^{(L-1)}_{lk}} (\overline{y_{tl}}-\overline{d_{tl}}) + (y_{tl}-d_{tl}) \frac{\partial \overline{y_{tl}}}{\partial w^{(L-1)}_{lk}}\right]\\
&=\frac{1}{N}\sum_{t=1}^N \left( \frac{\partial y_{tl}}{\partial y_{tl}^{\textrm{net}}}\frac{\partial y_{tl}^{\textrm{net}}}{\partial w^{(L-1)}_{lk}}  +\frac{\partial y_{tl}}{\partial \overline{y_{tl}^{\textrm{net}}}}\frac{\partial \overline{y_{tl}^{\textrm{net}}}}{\partial w^{(L-1)}_{lk}}\right) (\overline{y_{tl}}-\overline{d_{tl}})\\
&=\frac{1}{N}\sum_{t=1}^N (\overline{y_{tl}}-\overline{d_{tl}}) g'_L\left( y_{tl}^{\textrm{net}}\right) x^{(L-1)}_{tk},
\end{split}
\label{eq:gradupdateder}
\end{equation*}
so that
\begin{equation*}
\left( \frac{\partial E}{\partial w^{(L-1)}_{lk}} \right)^* = \frac{1}{N}\sum_{t=1}^N (y_{tl}-d_{tl}) g'_L\left( \overline{y_{tl}^{\textrm{net}}}\right) \overline{x^{(L-1)}_{tk}}.
\end{equation*}
The partial derivatives $\left( \frac{\partial E}{\partial w^{(p-1)}_{ji}}\right)^*$, taken with respect to the hidden layer weights $w^{(p-1)}_{ji}$, $1\le p \le L-1$, $j=1,...,K_p$, $i=1,...,K_{p-1}$, are computed recursively.  The partial derivatives $\frac{\partial E}{\partial w^{(L-2)}_{ji}}$, taken with respect to the $(L-2)$th hidden layer weights, are computed using the updated $(L-1)$th output layer weights:
\begin{equation*}
\begin{split}
\frac{\partial E}{\partial w^{(L-2)}_{ji}} &= \frac{\partial}{\partial w^{(L-2)}_{ji}} \left[ \frac{1}{N}\sum_{t=1}^N \sum_{l=1}^C (y_{tl}-d_{tl})(\overline{y_{tl}}-\overline{y_{tl}}) \right]\\
&= \frac{1}{N} \sum_{t=1}^N \sum_{l=1}^C \left[ \frac{\partial y_{tl}}{\partial w^{(L-2)}_{ji}} (\overline{y_{tl}}-\overline{d_{tl}}) + (y_{tl}-d_{tl}) \frac{\partial \overline{y_{tl}}}{\partial w^{(L-2)}_{ji}} \right],
\end{split}
\label{eq:gradupdateder2}
\end{equation*}
where
\begin{equation*}
\begin{split}
\frac{\partial y_{tl}}{\partial w^{(L-2)}_{ji}} &= \frac{\partial y_{tl}}{\partial y_{tl}^{\textrm{net}}} \frac{\partial y_{tl}^{\textrm{net}}}{\partial w^{(L-2)}_{ji}} + \frac{\partial y_{tl}}{\partial \overline{y_{tl}^{\textrm{net}}}} \frac{\overline{\partial y_{tl}^{\textrm{net}}}}{\partial w^{(L-2)}_{ji}}\\
&=g'_L\left( y_{tl}^{\textrm{net}}\right)\left( \frac{\partial y_{tl}^{\textrm{net}}}{\partial x_{tj}^{(L-1)}}\frac{\partial x_{tj}^{(L-1)}}{\partial w^{(L-2)}_{ji}} + \frac{\partial y_{tl}^{\textrm{net}}}{\partial \overline{x_{tj}^{(L-1)}}}\frac{\partial \overline{x_{tj}^{(L-1)}}}{\partial w^{(L-2)}_{ji}} \right)\\
&= g'_L \left( y_{tl}^{\textrm{net}}\right) w^{(L-1)}_{lj} \left( \frac{\partial x_{tj}^{(L-1)}}{\partial \left( x_{tj}^{(L-1)}\right)^{\textrm{net}}} \frac{\partial \left( x_{tj}^{(L-1)}\right)^{\textrm{net}}}{\partial w^{(L-2)}_{ji}}\right.\\
&\hspace{50mm}\left.+ \frac{\partial x_{tj}^{(L-1)}}{\partial \overline{\left( x_{tj}^{(L-1)}\right)^{\textrm{net}}}} \frac{\partial \overline{\left( x_{tj}^{(L-1)}\right)^{\textrm{net}}}}{\partial w^{(L-2)}_{ji}} \right)\\
&= g'_L \left( y_{tl}^{\textrm{net}}\right) w_{lj}^{(L-1)} g'_{L-1}\left( \left( x_{tj}^{(L-1)}\right)^{\textrm{net}}\right) x_{ti}^{(L-2)}
\end{split}
\end{equation*}
and
\begin{equation*}
\begin{split}
\frac{\partial \overline{y_{tl}}}{\partial w_{ji}^{(L-2)}} &= \frac{\partial \overline{y_{tl}}}{\partial \overline{y_{tl}^{\textrm{net}}}} \frac{\overline{\partial y_{tl}^{\textrm{net}}}}{\partial w^{(L-2)}_{ji}} + \frac{\partial \overline{y_{tl}}}{\partial y_{tl}^{\textrm{net}}} \frac{\partial y_{tl}^{\textrm{net}}}{\partial w^{(L-2)}_{ji}}\\
&=g'_L\left( \overline{y_{tl}^{\textrm{net}}}\right)\left( \frac{\partial \overline{y_{tl}^{\textrm{net}}}}{\partial \overline{x_{tj}^{(L-1)}}}\frac{\partial \overline{x_{tj}^{(L-1)}}}{\partial w^{(L-2)}_{ji}} + \frac{\partial \overline{y_{tl}^{\textrm{net}}}}{\partial x_{tj}^{(L-1)}}\frac{\partial x_{tj}^{(L-1)}}{\partial w^{(L-2)}_{ji}} \right)\\
&=g'_L \left( \overline{y_{tl}^{\textrm{net}}}\right) \overline{w^{(L-1)}_{lj}} \left( \frac{\partial \overline{x_{tj}^{(L-1)}}}{\partial \overline{\left( x_{tj}^{(L-1)}\right)^{\textrm{net}}}} \frac{\partial \overline{\left( x_{tj}^{(L-1)}\right)^{\textrm{net}}}}{\partial w^{(L-2)}_{ji}} \right.\\
&\hspace{50mm} \left.+ \frac{\partial \overline{x_{tj}^{(L-1)}}}{\partial \left( x_{tj}^{(L-1)}\right)^{\textrm{net}}} \frac{\partial \left( x_{tj}^{(L-1)}\right)^{\textrm{net}}}{\partial w^{(L-2)}_{ji}} \right)\\
&=0,
\end{split}
\end{equation*}
so that
\begin{equation*}
\begin{split}
\left(\frac{\partial E}{\partial w_{ji}^{(L-2)}} \right)^*&=\frac{1}{N}\sum_{t=1}^N \left( \sum_{l=1}^C (y_{tl}-d_{tl}) g'_L\left( \overline{y_{tl}^{\textrm{net}}}\right) \overline{w_{lj}^{(L-1)}}\right) \\
& \hspace{35mm} \cdot g'_{L-1}\left( \overline{\left( x_{tj}^{(L-1)}\right)^{\textrm{net}}}\right) \overline{x_{ti}^{(L-2)}},
\end{split}
\end{equation*}
and so on.  We summarize the partial derivatives by
\bea{\left( \frac{\partial E}{\partial w^{(p-1)}_{ji}}\right)^* = \frac{1}{N} \sum_{t=1}^N E^{(p)}_{tj} \overline{x^{(p-1)}_{ti}}}, \label{eq:grad}\eea
$1\le p\le L$, where $j=1,...,K_p$, $i=1,...,K_{p-1}$, and the $E_{tj}^{(p)}$ are given recursively by
\bea{E^{(L)}_{tl}=\left( y_{tl}-d_{tl}\right) g_L'\left( \overline{y^{\textrm{net}}_{tl}}\right), \label{eq:deltaL}}\eea
where $l=1,...,C$, $t=1,...,N$; and for $1\le p\le L-1$,
\bea{E^{(p)}_{tj} = \left[ \sum_{\alpha =1}^{K_{p+1}} E^{(p+1)}_{t\alpha} \overline{w^{(p)}_{\alpha j}}\right] g_p'\left(\overline{\left(x^{(p)}_{tj}\right)^{\textrm{net}}}\right), \label{eq:deltap}}\eea
where $j=1,...,K_p$, $t=1,...,N$.  The gradient descent method is well known to be rather slow in the convergence of the error function.  We next derive formulas for the backpropagation algorithm using Newton's method (compare with \cite{Buchholz2008,Leung1991}).

\section{Backpropagation Using Newton's Method}

\par

The weight updates for Newton's method with complex functions are given by formula (111) of \cite{Kreutz-Delgado2009} (we omit the superscripts, which index the layers, to simplify our writing): 
\begin{equation}
\Delta\mathbf{w} = \left( \mathcal{H}_{\mathbf{w}\mathbf{w}} - \mathcal{H}_{\overline{\mathbf{w}}\mathbf{w}}\mathcal{H}_{\overline{\mathbf{w}}\overline{\mathbf{w}}}^{-1}\mathcal{H}_{\mathbf{w}\overline{\mathbf{w}}} \right)^{-1}
\left[ \mathcal{H}_{\overline{\mathbf{w}}\mathbf{w}}\mathcal{H}_{\overline{\mathbf{w}}\overline{\mathbf{w}}}^{-1} \left( \frac{\partial E}{\partial \overline{\mathbf{w}}} \right)^* - \left( \frac{\partial E}{\partial \mathbf{w}} \right)^* \right].
\label{eq:NewUp} \end{equation}

\par
To apply the Newton algorithm we need to compute the Hessian matrices (again omitting the superscripts) 
\begin{equation}
\mathcal{H}_{\mathbf{w}\mathbf{w}}=\frac{\partial}{\partial \mathbf{w}} \left( \frac{\partial E}{\partial \mathbf{w}}\right)^* \textrm{ and } \mathcal{H}_{\overline{\mathbf{w}}\mathbf{w}}=\frac{\partial}{\partial \overline{\mathbf{w}}} \left( \frac{\partial E}{\partial \mathbf{w}}\right)^*,
\label{eq:hessdef}
\end{equation}
where the entries of $\left(\frac{\partial E}{\partial \mathbf{w}}\right)^*$ are given by (\ref{eq:grad}).  Note that although (\ref{eq:NewUp}) asks for the four Hessian matrices $\mathcal{H}_{\mathbf{w}\mathbf{w}}$, $\mathcal{H}_{\overline{\mathbf{w}}\mathbf{w}}$, $\mathcal{H}_{\mathbf{w}\overline{\mathbf{w}}}$, and $\mathcal{H}_{\overline{\mathbf{w}}\overline{\mathbf{w}}}$, we have $\mathcal{H}_{\mathbf{w}\overline{\mathbf{w}}} =\overline{\mathcal{H}_{\overline{\mathbf{w}}\mathbf{w}}} \textrm{ and } \mathcal{H}_{\overline{\mathbf{w}}\overline{\mathbf{w}}}=\overline{\mathcal{H}_{\mathbf{w}\mathbf{w}}}.$  Thus we only need to compute two of them.

\par
We consider the entries of the Hessian matrices $\mathcal{H}_{\mathbf{w}\mathbf{w}}$ and $\mathcal{H}_{\overline{\mathbf{w}}\mathbf{w}}$. For the $(p-1)$th layer, the entries of $\mathcal{H}_{\mathbf{w}\mathbf{w}}$ are given by (see (\ref{eq:weightcomponents}))
\begin{equation*}
\begin{split}
\mathcal{H}_{\mathbf{w}\mathbf{w}} \left[ (j-1) \cdot K_{p-1}+i  ,  (b-1)\cdot K_{p-1}+a\right] =\frac{\partial}{\partial w^{(p-1)}_{ba}}\left( \frac{\partial E}{\partial w^{(p-1)}_{ji}}\right)^*,
\end{split}
\end{equation*}
where $j,b=1,...,K_p$ and $i,a=1,...,K_{p-1}$, and the entries of  $\mathcal{H}_{\overline{\mathbf{w}}\mathbf{w}}$ are given by
\begin{equation*}
\begin{split}
\mathcal{H}_{\overline{\mathbf{w}}\mathbf{w}} \left[ (j-1) \cdot K_{p-1}+i  ,  (b-1)\cdot K_{p-1}+a\right]
= \frac{\partial}{\partial \overline{w^{(p-1)}_{ba}}}\left( \frac{\partial E}{\partial w^{(p-1)}_{ji}}\right)^*,
\end{split}
\end{equation*}
where $j,b=1,...,K_p$ and $i,a=1,...,K_{p-1}$.

\par
First we derive an explicit formula for the entries of the Hessians $\mathcal{H}_{\mathbf{w}\mathbf{w}}$.  We start with the output layer and compute $\frac{\partial}{\partial w^{(L-1)}_{kq}}\left( \frac{\partial E}{\partial w^{(L-1)}_{lp}}\right)^*$, where $k,l=1,...,C$ and $q,p=1,...,K_{L-1}$.  Observe that if $k\neq l$, then each term $(y_{tl}-d_{tl})g_L'\left(\overline{y^{\textrm{net}}_{tl}}\right)\overline{x^{(L-1)}_{tp}}$ in the cogradient given by (\ref{eq:grad}) and (\ref{eq:deltaL}) does not depend on the weights $w^{(L-1)}_{kq}$, hence this entry of the Hessian will be $0$.  So the Hessian matrix for the output layer has a block diagonal form:
\begin{equation*}\mathcal{H}_{\mathbf{w}^{(L-1)}\mathbf{w}^{(L-1)}}=\textrm{diag} \left\{ \left[ \frac{\partial}{\partial w^{(L-1)}_{lq}}\left( \frac{\partial E}{\partial w^{(L-1)}_{lp}} \right)^* \right]_{1\le p\le K_{L-1} \atop 1\le q\le K_{L-1}} : l=1,...,C\right\}.\end{equation*}
Now:
\begin{equation}
\begin{split}
\frac{\partial}{\partial w^{(L-1)}_{lq}} &\left(  \frac{\partial E}{\partial w^{(L-1)}_{lp}} \right)^* = \frac{\partial}{\partial w^{(L-1)}_{lq}} \left[ \frac{1}{N} \sum_{t=1}^N (y_{tl}-d_{tl}) g_L'\left( \overline{y^{\textrm{net}}_{tl}} \right) \overline{x^{(L-1)}_{tp}} \right]\\
&= \frac{1}{N}\sum_{t=1}^N \left[ (y_{tl}-d_{tl}) \frac{\partial g_L'\left( \overline{y^{\textrm{net}}_{tl}}\right)}{\partial w^{(L-1)}_{lq}} + g_L'\left( \overline{y^{\textrm{net}}_{tl}} \right) \frac{\partial y_{tl}}{\partial w^{(L-1)}_{lq} } \right] \overline{x^{(L-1)}_{tp}}
\end{split}\label{eq:outcomp}
\end{equation}
where
\begin{equation*}
\frac{\partial y_{tl}}{\partial w^{(L-1)}_{lq}} = \frac{\partial y_{tl}}{\partial y^{\textrm{net}}_{tl}}\frac{\partial y^{\textrm{net}}_{tl}}{\partial w^{(L-1)}_{lq}}+\frac{\partial y_{tl}}{\partial \overline{y^{\textrm{net}}_{tl}}}\frac{\partial \overline{y^{\textrm{net}}_{tl}}}{\partial w^{(L-1)}_{lq}}=g_L'\left( y^{\textrm{net}}_{tl}\right) x^{(L-1)}_{tq}
\end{equation*}
since $g_L$ is holomorphic and therefore $\frac{\partial y_{tl}}{\partial \overline{y_{tl}^{\textrm{net}}}}=0$ (Cauchy-Riemann condition), and similarly
\begin{equation*}
\frac{\partial g_L'\left( \overline{y^{\textrm{net}}_{tl}}\right)}{\partial w^{(L-1)}_{lq}} = \frac{\partial g_L'\left( \overline{y^{\textrm{net}}_{tl}}\right)}{\partial \overline{y^{\textrm{net}}_{tl}}}\frac{\overline{\partial y^{\textrm{net}}_{tl}}}{\partial w^{(L-1)}_{lq}}+\frac{\partial g_L'\left( \overline{y^{\textrm{net}}_{tl}}\right)}{\partial y^{\textrm{net}}_{tl}}\frac{\partial y^{\textrm{net}}_{tl}}{\partial w^{(L-1)}_{lq}}=0.
\end{equation*}
Combining these two partial derivatives with (\ref{eq:outcomp}) gives the following formula for the entries of the output layer Hessian matrix:
\begin{equation}
\begin{split}
\frac{\partial}{\partial w^{(L-1)}_{kq}}&\left( \frac{\partial E}{\partial w^{(L-1)}_{lp}}\right)^*\\
&=\left\{ \begin{array}{ll}
\frac{1}{N}\sum_{t=1}^N g_L'\left( \overline{y^{\textrm{net}}_{tl}} \right)g_L'\left( y^{\textrm{net}}_{tl} \right) \overline{x^{(L-1)}_{tp}}x^{(L-1)}_{tq} &  \textrm{if }k=l,\\
0 &  \textrm{if }k\neq l.
\end{array}\right.
\end{split}
\label{eq:outputhess}
\end{equation}

\par
After updating the output layer weights, the backpropagation algorithm updates the hidden layer weights recursively.  We compute the entries of the Hessian  $\mathcal{H}_{\mathbf{w}^{(p-1)}\mathbf{w}^{(p-1)}}$ for the $(p-1)$th layer using (\ref{eq:grad}):
\begin{equation}
\begin{split}
\frac{\partial}{\partial w^{(p-1)}_{ba}}\left( \frac{\partial E}{\partial w^{(p-1)}_{ji}}\right)^* &= \frac{\partial}{\partial w^{(p-1)}_{ba}} \left[ \frac{1}{N}\sum_{t=1}^N E^{(p)}_{tj} \overline{x^{(p-1)}_{ti}} \right]\\
&= \frac{1}{N}\sum_{t=1}^N \frac{\partial E^{(p)}_{tj}}{\partial w^{(p-1)}_{ba}} \overline{x^{(p-1)}_{ti}}.
\end{split}
\label{eq:hidcomp}
\end{equation}
Applying the chain rule to (\ref{eq:deltap}), we have
\begin{equation}
\begin{split}
&\frac{\partial E^{(p)}_{tj}}{\partial w^{(p-1)}_{ba}} = \frac{\partial}{\partial w^{(p-1)}_{ba}}\left[\left( \sum_{\eta =1}^{K_{p+1}} E^{(p+1)}_{t\eta} \overline{w^{(p)}_{\eta j}}\right) g_p'\left(\overline{\left(x^{(p)}_{tj}\right)^{\textrm{net}}}\right)\right]\\
&= g_p'\left(\overline{\left(x^{(p)}_{tj}\right)^{\textrm{net}}}\right) \sum_{\eta =1}^{K_{p+1}} \frac{\partial E^{(p+1)}_{t\eta}}{\partial w^{(p-1)}_{ba}} \overline{w^{(p)}_{\eta j}}\\
&=g_p'\left(\overline{\left(x^{(p)}_{tj}\right)^{\textrm{net}}}\right) \sum_{\eta =1}^{K_{p+1}} \left[ \frac{\partial E^{(p+1)}_{t\eta}}{\partial x^{(p)}_{tb}} \frac{\partial x^{(p)}_{tb}}{\partial w^{(p-1)}_{ba}}\right.+ \left.\frac{\partial E^{(p+1)}_{t\eta}}{\partial \overline{x^{(p)}_{tb}}} \frac{\partial \overline{x^{(p)}_{tb}}}{\partial w^{(p-1)}_{ba}} \right] \overline{w^{(p)}_{\eta j}}\\
&= g_p'\left(\overline{\left(x^{(p)}_{tj}\right)^{\textrm{net}}}\right) \sum_{\eta =1}^{K_{p+1}} \frac{\partial E^{(p+1)}_{t\eta}}{\partial x^{(p)}_{tb}}  \left[ \frac{\partial x^{(p)}_{tb}}{\partial (x^{(p)}_{tb})^{\textrm{net}}}  \frac{\partial (x^{(p)}_{tb})^{\textrm{net}}}{\partial w^{(p-1)}_{ba}}\right.\\
& \hspace{40mm} \left.+ \frac{\partial x^{(p)}_{tb}}{\partial \overline{(x^{(p)}_{tb})^{\textrm{net}}}} \frac{\partial \overline{(x^{(p)}_{tb})^{\textrm{net}}}}{\partial w^{(p-1)}_{ba}} \right] \overline{w^{(p)}_{\eta j}}\\
&= g_p'\left(\overline{\left(x^{(p)}_{tj}\right)^{\textrm{net}}}\right)\sum_{\eta =1}^{K_{p+1}}\frac{\partial E^{(p+1)}_{t\eta}}{\partial x^{(p)}_{tb}} g_p'\left(\left(x^{(p)}_{tb}\right)^{\textrm{net}}\right) x^{(p-1)}_{ta} \overline{w^{(p)}_{\eta j}}.
\end{split}
\label{eq:deltaw}
\end{equation}
In the above computation, we have used the fact that $g_p$ is holomorphic and hence $\frac{\partial g'_p\left( \overline{( x^{(p)}_{tj})^{\textrm{net}}}\right)}{\partial w^{(p-1)}_{ba}}=0$ and $\frac{\partial \overline{x_{tb}^{(p)}}}{\partial w_{ba}^{(p-1)}}=0$,  $\frac{\partial x^{(p)}_{tb}}{\partial \overline{( x^{(p)}_{tb})^{\textrm{net}}}}=0$, and $\frac{\partial \overline{ (x^{(p)}_{tb})^{\textrm{net}}}}{\partial w^{(p-1)}_{ba}} =0$.  Combining  (\ref{eq:hidcomp}) and (\ref{eq:deltaw}), we have:
\begin{equation}
\begin{split}
\frac{\partial}{\partial w^{(p-1)}_{ba}}\left( \frac{\partial E}{\partial w^{(p-1)}_{ji}}\right)^* &=\frac{1}{N}\sum_{t=1}^N  \left[ \sum_{\eta=1}^{K_{p+1}} \frac{\partial E^{(p+1)}_{t\eta}}{\partial x^{(p)}_{tb}} \overline{w^{(p)}_{\eta j}} \right]\\ & \hspace{10mm} \cdot g_p'\left(\overline{\left(x^{(p)}_{tj}\right)^{\textrm{net}}}\right) g_p'\left(\left(x^{(p)}_{tb}\right)^{\textrm{net}}\right) \overline{x^{(p-1)}_{ti}} x^{(p-1)}_{ta}.
\end{split}
\label{eq:hiddeltas}
\end{equation}

\par
Next, we derive a recursive rule for finding the partial derivatives $\frac{\partial E^{(p+1)}_{t\eta}}{\partial x^{(p)}_{tb}}$.  For computational purposes, an explicit formula for $\frac{\partial E^{(L)}_{t\eta}}{\partial x^{(L-1)}_{tb}}$ is not necessary.  What we need is a recursive formula for these partial derivatives as will be apparent shortly.  Using (\ref{eq:deltap}) we have the following:
\begin{equation}
\begin{split}
\frac{\partial E^{(p+1)}_{t\eta}}{\partial x^{(p)}_{tb}} &= \frac{\partial}{\partial x^{(p)}_{tb}} \left[\left( \sum_{\alpha =1}^{K_{p+2}} E^{(p+2)}_{t\alpha} \overline{w^{(p+1)}_{\alpha \eta}}\right) g_{p+1}'\left(\overline{\left(x^{(p+1)}_{t\eta}\right)^{\textrm{net}}}\right)\right]\\
&=g_{p+1}'\left(\overline{\left(x^{(p+1)}_{t\eta}\right)^{\textrm{net}}}\right) \sum_{\alpha =1}^{K_{p+2}} \frac{\partial E^{(p+2)}_{t\alpha}}{\partial x^{(p)}_{tb}} \overline{w^{(p+1)}_{\alpha \eta}}\\
&=g_{p+1}'\left(\overline{\left(x^{(p+1)}_{t\eta}\right)^{\textrm{net}}}\right) \sum_{\alpha =1}^{K_{p+2}}\sum_{\beta=1}^{K_{p+1}} \left[ \frac{\partial E^{(p+2)}_{t\alpha}}{\partial x^{(p+1)}_{t\beta}} \frac{\partial x^{(p+1)}_{t\beta}}{\partial x^{(p)}_{tb}}\right.\\
& \hspace{40mm} \left.+\frac{\partial E^{(p+2)}_{t\alpha}}{\partial \overline{x^{(p+1)}_{t\beta}}} \frac{\partial \overline{x^{(p+1)}_{t\beta}}}{\partial x^{(p)}_{tb}} \right] \overline{w^{(p+1)}_{\alpha \eta}}\\
\end{split}
\label{eq:deltapbyx}
\end{equation}
\begin{equation*}
\begin{split}
&=g_{p+1}'\left(\overline{\left(x^{(p+1)}_{t\eta}\right)^{\textrm{net}}}\right) \sum_{\alpha =1}^{K_{p+2}}\sum_{\beta=1}^{K_{p+1}}\frac{\partial E^{(p+2)}_{t\alpha}}{\partial x^{(p+1)}_{t\beta}}\\
& \hspace{3mm} \cdot\left[ \frac{\partial x^{(p+1)}_{t\beta}}{\partial (x^{(p+1)}_{t\beta})^{\textrm{net}}} \frac{\partial (x^{(p+1)}_{t\beta})^{\textrm{net}}}{\partial x^{(p)}_{tb}}+\frac{\partial x^{(p+1)}_{t\beta}}{\partial \overline{(x^{(p+1)}_{t\beta})^{\textrm{net}}}} \frac{\partial \overline{(x^{(p+1)}_{t\beta})^{\textrm{net}}}}{\partial x^{(p)}_{tb}} \right] \overline{w^{(p+1)}_{\alpha \eta}}\\
&=\sum_{\beta=1}^{K_{p+1}}\left[\sum_{\alpha =1}^{K_{p+2}}\frac{\partial E^{(p+2)}_{t\alpha}}{\partial x^{(p+1)}_{t\beta}} \overline{w^{(p+1)}_{\alpha \eta}}\right]\\
& \hspace{20mm} \cdot  g_{p+1}'\left(\overline{\left(x^{(p+1)}_{t\eta}\right)^{\textrm{net}}}\right)g_{p+1}'\left(\left(x^{(p+1)}_{t\beta}\right)^{\textrm{net}}\right)w^{(p)}_{\beta b}.
\end{split}
\end{equation*}
This gives a recursive formula for computing the partial derivatives $\frac{\partial E^{(p+1)}_{t\eta}}{\partial x^{(p)}_{tb}}$.  We will combine the above calculations to give a more concise recursive algorithm for computing the entries of the matrices $\mathcal{H}_{\mathbf{w}\mathbf{w}}$ in Theorem \ref{thm:complexnewtbackprop}, below.

\par

Next we consider the Hessians $\mathcal{H}_{\overline{\mathbf{w}}\mathbf{w}}$.  Again we start with the output layer and compute $\frac{\partial}{\partial \overline{w^{(L-1)}_{kq}}} \left( \frac{\partial E}{\partial w^{(L-1)}_{lp}} \right)^*$.  Using the fact that $\frac{\partial E}{\partial w^{(L-1)}_{lp}}$ does not depend on $\overline{w^{(L-1)}_{kq}}$ if $k\neq l$, we see that the output layer Hessian $\mathcal{H}_{\overline{\mathbf{w}^{(L-1)}}\mathbf{w}^{(L-1)}}$ is also block diagonal with blocks
\begin{eqnarray*}
\left[ \frac{\partial}{\partial \overline{w^{(L-1)}_{lq}}} \left( \frac{\partial E}{\partial w^{(L-1)}_{lp}} \right)^*\right]_{1\le p\le K_{L-1} \atop 1\le q\le K_{L-1}}
\end{eqnarray*}
for $l=1,...,C$.  Computing the entries in these blocks,
\begin{equation*}
\begin{split}
\frac{\partial}{\partial \overline{w^{(L-1)}_{lq}}} & \left( \frac{\partial E}{\partial w^{(L-1)}_{lp}} \right)^* = \frac{\partial}{\partial \overline{w^{(L-1)}_{lq}}} \left[ \frac{1}{N} \sum_{t=1}^N (y_{tl}-d_{tl}) g_L'\left(\overline{y^{\textrm{net}}_{tl}}\right) \overline{x^{(L-1)}_{tp}}\right]\\
&= \frac{1}{N} \sum_{t=1}^N \left[ (y_{tl}-d_{tl})\frac{\partial g_L'\left(\overline{y^{\textrm{net}}_{tl}}\right)}{\partial \overline{w^{(L-1)}_{lq}}}+ g_L'\left(\overline{y^{\textrm{net}}_{tl}}\right) \frac{\partial y_{tl}}{\partial \overline{w^{(L-1)}_{lq}}} \right]\overline{x^{(L-1)}_{tp}}
\end{split}
\label{eq:outconjcomp}
\end{equation*}
where $\frac{\partial y_{tl}}{\partial \overline{w^{(L-1)}_{lq}}} =0$, and
\begin{eqnarray*}
\begin{split}
\frac{\partial g_L'\left(\overline{y^{\textrm{net}}_{tl}}\right)}{\partial \overline{w^{(L-1)}_{lq}}} &=\frac{\partial g_L'\left(\overline{y^{\textrm{net}}_{tl}}\right)}{\partial \overline{y_{tl}^{\textrm{net}}}} \frac{\partial \overline{y_{tl}^{\textrm{net}}}}{\partial \overline{w^{(L-1)}_{lq}}}+\frac{\partial g_L'\left(\overline{y^{\textrm{net}}_{tl}}\right)}{\partial y_{tl}^{\textrm{net}}} \frac{\partial y_{tl}^{\textrm{net}}}{\partial \overline{w^{(L-1)}_{lq}}}\\
&= g_L''\left(\overline{y_{tl}^{\textrm{net}}}\right)\overline{x^{(L-1)}_{tq}}.
\end{split}
\end{eqnarray*}
Thus:
\begin{equation}
\begin{split}
\frac{\partial}{\partial \overline{w^{(L-1)}_{kq}}} & \left( \frac{\partial E}{\partial w^{(L-1)}_{lp}} \right)^*\\ &=\left\{ \begin{array}{ll} \frac{1}{N}\sum_{t=1}^N (y_{tl}-d_{tl}) g_L''\left(\overline{y_{tl}^{\textrm{net}}}\right)\overline{x^{(L-1)}_{tq}}\overline{x^{(L-1)}_{tp}} & \textrm{if } k=l,\\
0 & \textrm{if }k\neq l. \end{array}\right.
\end{split}
\label{eq:outconj}
\end{equation}

\par

The entries of the Hessian $\mathcal{H}_{\overline{\mathbf{w}^{(p-1)}}\mathbf{w}^{(p-1)}}$ for the $(p-1)$th layer can be computed similarly.  We record the formula here and provide the detailed computations in Appendix A.  We have:
\begin{equation}
\begin{split}
&\frac{\partial}{\partial \overline{w^{(p-1)}_{ba}}} \left( \frac{\partial E}{\partial w^{(p-1)}_{ji}} \right)^*\\
&=\left\{ \begin{array}{l} 
\frac{1}{N} \sum_{t=1}^N \left\{ \left[ \sum_{\eta=1}^{K_{p+1}}\frac{\partial E^{(p+1)}_{t\eta}}{\partial \overline{x^{(p)}_{tb}}} \overline{w^{(p)}_{\eta j}} \right] g_p'( \overline{(x^{(p)}_{tj})^{\textrm{net}}})g_p'( \overline{(x^{(p)}_{tb})^{\textrm{net}}}) \right.\\
\hspace{10mm} \left. +\left[\sum_{\eta=1}^{K_{p+1}} E_{t\eta}^{(p+1)} w^{(p)}_{\eta j} \right] g_p''( \overline{(x^{(p)}_{tj})^{\textrm{net}}}) \right\} \overline{x^{(p-1)}_{ti}}\overline{x^{(p-1)}_{ta}}\textrm{ if }j=b,\\
\frac{1}{N} \sum_{t=1}^N \left\{ \left[ \sum_{\eta=1}^{K_{p+1}}\frac{\partial E^{(p+1)}_{t\eta}}{\partial \overline{x^{(p)}_{tb}}} \overline{w^{(p)}_{\eta j}} \right] g_p'( \overline{(x^{(p)}_{tj})^{\textrm{net}}})g_p'( \overline{(x^{(p)}_{tb})^{\textrm{net}}}) \right\}\\
\hspace{70mm}\cdot \overline{x^{(p-1)}_{ti}}\overline{x^{(p-1)}_{ta}} \hspace{3.5mm}\textrm{if }j\neq b,\\
\end{array}\right.
\end{split}
\label{eq:hidconjdeltas}
\end{equation}
where $j,b=1,...,K_p$ and $i,a=1,...,K_{p+1}$, and the partial derivatives $\frac{\partial E^{(p+1)}_{t\eta}}{\partial \overline{x^{(p)}_{tb}}}$ are given recursively by
\begin{equation}
\begin{split}
\frac{\partial E^{(p+1)}_{t\eta}}{\partial \overline{x^{(p)}_{tb}}} &=\sum_{\beta=1}^{K_{p+1}} \left[ \sum_{\alpha=1}^{K_{p+2}} \frac{\partial E^{(p+2)}_{t\alpha}}{\partial \overline{x^{(p+1)}_{t\beta}}}\overline{w^{(p+1)}_{\alpha \eta}} \right]\\
&\hspace{20mm}\cdot g'_{p+1}\left(\overline{\left( x^{(p+1)}_{t\eta}\right)^{\textrm{net}}}\right) g'_{p+1}\left(\overline{\left( x^{(p+1)}_{t\beta}\right)^{\textrm{net}}}\right) \overline{w^{(p)}_{\beta b}}\\
&\hspace{10mm}+\left[\sum_{\alpha=1}^{K_{p+2}} E^{(p+2)}_{t\alpha} \overline{w^{(p+1)}_{\alpha \eta}} \right]g''_{p+1}\left(\overline{\left( x^{(p+1)}_{t\eta}\right)^{\textrm{net}}}\right) \overline{w^{(p)}_{\eta b}}.
\end{split}
\label{eq:deltapbyxconj}
\end{equation}

\par
We now summarize the formulas we have derived in the following theorem.

\begin{theorem}[Newton Backpropagation Algorithm for Holomorphic Neural Networks] The weight updates for the holomorphic MLPs with activation functions satisfying $$\overline{g(z)}=g(\overline{z}),$$ $p=1,...,L$, using the backpropagation algorithm with Newton's method are given by
\begin{equation}
\begin{split}
\Delta\mathbf{w}^{(p-1)}
& = \left( \mathcal{H}_{\mathbf{w}^{(p-1)}\mathbf{w}^{(p-1)}} - \mathcal{H}_{\overline{\mathbf{w}^{(p-1)}}\mathbf{w}^{(p-1)}}\mathcal{H}_{\overline{\mathbf{w}^{(p-1)}}\overline{\mathbf{w}^{(p-1)}}}^{-1}\mathcal{H}_{\mathbf{w}^{(p-1)}\overline{\mathbf{w}^{(p-1)}}} \right)^{-1}\\
&\hspace{5mm}\cdot\left[ \mathcal{H}_{\overline{\mathbf{w}^{(p-1)}}\mathbf{w}^{(p-1)}}\mathcal{H}_{\overline{\mathbf{w}^{(p-1)}}\overline{\mathbf{w}^{(p-1)}}}^{-1} \left( \frac{\partial E}{\partial \overline{\mathbf{w}^{(p-1)}}} \right)^* - \left( \frac{\partial E}{\partial \mathbf{w}^{(p-1)}} \right)^* \right],
\end{split}
\label{eq:NewUpThm}
\end{equation}
where:
\begin{enumerate}
\item the entries of the Hessian matrices $\mathcal{H}_{\mathbf{w}^{(p-1)}\mathbf{w}^{(p-1)}}$ for $p=1,...,L$ are given by
\begin{equation}
\frac{\partial}{\partial w^{(p-1)}_{ba}}\left( \frac{\partial E}{\partial w^{(p-1)}_{ji}}\right)^* = \frac{1}{N} \sum_{t=1}^N \gamma^{(p)}_{tjb} \overline{x^{(p-1)}_{ti}} x^{(p-1)}_{ta}
\label{eq:hessp}
\end{equation}
for $j,b=1,...,K_p$ and $i,a=1,...,K_{p-1}$, where the $\gamma^{(p)}_{tjb}$ are defined for $t=1,...,N$ recursively on $p$ by
\begin{equation*}
\gamma^{(L)}_{tkl} = \left\{ \begin{array}{ll} 
g_L'(\overline{y^{\textrm{net}}_{tl}})g_L'(y^{\textrm{net}}_{tl}) & \textrm{if } k=l,\\
0 & \textrm{if } k\neq l,   \end{array}\right.
\label{eq:gammaL}
\end{equation*}
for $k,l=1,...,C$, and for $p=1,...,L-1$,
\begin{equation}
\gamma^{(p)}_{tjb} = \left[ \sum_{\eta=1}^{K_{p+1}} \sum_{\beta=1}^{K_{p+1}} \gamma^{(p+1)}_{t\eta \beta} \overline{w^{(p)}_{\eta j}}w^{(p)}_{\beta b}\right] g_p'\left( \overline{(x^{(p)}_{tj})^{\textrm{net}}} \right)g_p'\left( (x^{(p)}_{tb})^{\textrm{net}} \right)
\label{eq:gammap}
\end{equation}
for $j,b=1,...,K_{p+1}$,

\item the entries of the Hessian matrices $\mathcal{H}_{\overline{\mathbf{w}^{(p-1)}}\mathbf{w}^{(p-1)}}$ for $p=1,...,L$ are given by
\begin{equation}
\frac{\partial}{\partial \overline{w^{(p-1)}_{ba}}} \left( \frac{\partial E}{\partial w^{(p-1)}_{ji}}\right)^* = \frac{1}{N} \sum_{t=1}^N \left( \psi^{(p)}_{tjb} +\theta^{(p)}_{tjb}\right) \overline{x^{(p-1)}_{ti}}\overline{x^{(p-1)}_{ta}}
\label{eq:conjhessp}
\end{equation}
for $j,b=1,...,K_p$ and $i,a=1,...,K_{p-1}$, where the $\theta^{(p)}_{tjb}$ are defined for $t=1,...,N$ by 
\begin{equation*}
\theta^{(L)}_{tkl}=\left\{ \begin{array}{ll} (y_{tl}-d_{tl})g_L''\left(\overline{y_{tl}^{\textrm{net}}} \right) & \textrm{if }k=l, \\
0 & \textrm{if }k\neq l, \end{array}\right.
\label{thetaL}
\end{equation*}
for $k,l=1,...,C$, and for $p=1,...,L-1$, 
\begin{equation}
\theta^{(p)}_{tjb} =\left\{ \begin{array}{ll} \left[ \sum_{\eta=1}^{K_{p+1}} E^{(p+1)}_{t\eta} w^{(p)}_{\eta j} \right] g_p'' \left( \overline{\left( x^{(p)}_{tj}\right)^{\textrm{net}}}\right) & \textrm{if } j=b, \\
0 & \textrm{if } j\neq b, \end{array}\right.
\label{eq:thetap}
\end{equation}
for $j,b=1,...,K_{p+1}$, where the $E^{(p)}_{t\eta}$ are given by (\ref{eq:deltaL}) and (\ref{eq:deltap}), and the $\psi^{(p)}_{tjb}$ are defined for $t=1,...,N$ recursively on $p$ by $\psi^{(L)}_{tkl}=0$ for $k,l=1,...,C$, and for $p=1,...,L-1$, 
\begin{equation}
\begin{split}
\psi^{(p)}_{tjb}=\left[ \sum_{\eta=1}^{K_{p+1}}\sum_{\beta=1}^{K_{p+1}} \left( \psi^{(p+1)}_{t\eta\beta} \overline{w^{(p)}_{\beta b}}\right.\right. &+\left.\left.\theta^{(p+1)}_{t\eta\beta} \overline{w^{(p)}_{\eta b}}\right)\overline{w^{(p)}_{\eta j}}  \right]\\ &\cdot g_p'\left(\overline{\left(x^{(p)}_{tj}\right)^{\textrm{net}}} \right) g_p'\left(\overline{\left(x^{(p)}_{tb}\right)^{\textrm{net}}} \right)
\end{split}
\label{eq:psip}
\end{equation}
for $j,b=1,...,K_{p+1}$, and

\item for the other two Hessian matrices we have $\mathcal{H}_{\mathbf{w}^{(p-1)}\overline{\mathbf{w}^{(p-1)}}}=\overline{\mathcal{H}_{\overline{\mathbf{w}^{(p-1)}}\mathbf{w}^{(p-1)}}}$ and $\mathcal{H}_{\overline{\mathbf{w}^{(p-1)}}\overline{\mathbf{w}^{(p-1)}}} = \overline{\mathcal{H}_{\mathbf{w}^{(p-1)}\mathbf{w}^{(p-1)}}}.$
\end{enumerate}
\label{thm:complexnewtbackprop}
\end{theorem}

\begin{proof}\begin{enumerate} \item Setting $\gamma^{(L)}_{tkl}$ as defined above, Equation (\ref{eq:hessp}) follows immediately from (\ref{eq:outputhess}).  For the hidden layer Hessian matrix entries, set 
\begin{equation}
\gamma^{(p)}_{tjb} = \left[\sum_{\eta=1}^{K_{p+1}} \frac{\partial E^{(p+1)}_{t\eta}}{\partial x^{(p)}_{tb}} \overline{w^{(p)}_{\eta j}}\right] g_p'\left( \overline{(x^{(p)}_{tj})^{\textrm{net}}} \right)g_p'\left( (x^{(p)}_{tb})^{\textrm{net}} \right)
\label{eq:proof1a}
\end{equation}
in (\ref{eq:hiddeltas}), giving us (\ref{eq:hessp}).  Then using (\ref{eq:deltapbyx}) we have 
\begin{equation} 
\frac{\partial E_{t\eta}^{(p+1)}}{\partial x_{tb}^{(p)}} = \sum_{\beta=1}^{K_{p+1}} \gamma^{(p+1)}_{t\eta\beta} w^{(p)}_{\beta b}.
\label{eq:proof1b} \end{equation}
So substituting (\ref{eq:proof1b}) into (\ref{eq:proof1a}) we get the recursive formula (\ref{eq:gammap}).

\item The formula (\ref{eq:conjhessp}) for $p=L$ follows directly from  the way we defined  $\theta^{(L)}_{tkl}$, $\psi^{(L)}_{tkl}$, and equation (\ref{eq:outconj}).  Next, define the $\theta^{(p)}_{tjb}$ as above, and set
\begin{equation}
\psi^{(p)}_{tjb}=\left[ \sum_{\eta=1}^{K_{p+1}} \frac{\partial E^{(p+1)}_{t\eta}}{\partial \overline{x^{(p)}_{tb}}} \overline{w^{(p)}_{\eta j}}\right]g_p'\left(\overline{\left(x^{(p)}_{tj}\right)^{\textrm{net}}} \right)g_p'\left(\overline{\left(x^{(p)}_{tb}\right)^{\textrm{net}}} \right)
\label{eq:proof2a}
\end{equation}
in (\ref{eq:hidconjdeltas}).  Substituting (\ref{eq:proof2a}) and (\ref{eq:thetap}) into (\ref{eq:hidconjdeltas}) gives us (\ref{eq:conjhessp}). For the $\psi^{(p)}_{tjb}$, using (\ref{eq:deltapbyxconj}) with our definition of the $\psi^{(p)}_{tjb}$ in (\ref{eq:proof2a}) we have:
\begin{equation}
\frac{\partial E^{(p+1)}_{t\eta}}{\partial \overline{x^{(p)}_{tb}}}=\sum_{\beta=1}^{K_{p+1}} \left( \psi^{(p+1)}_{t\eta\beta} \overline{w^{(p)}_{\beta b}}+\theta^{(p+1)}_{t\eta\beta} \overline{w^{(p)}_{\eta b}}\right)
\label{eq:proof2b}
\end{equation}
so substituting (\ref{eq:proof2b}) into (\ref{eq:proof2a}) we get (\ref{eq:psip}).\end{enumerate} \end{proof}

\section{Backpropagation Using the Pseudo-Newton's Method}

To simplify the computation in the implementation of Newton's method, we can use the pseudo-Newton algorithm, which is an alternative algorithm also known to provide good quadratic convergence.  For the pseudo-Newton algorithm, we take $\mathcal{H}_{\overline{\mathbf{w}^{(p-1)}}\mathbf{w}^{(p-1)}}=0=\mathcal{H}_{\mathbf{w}^{(p-1)}\overline{\mathbf{w}^{(p-1)}}}$ in (\ref{eq:NewUpThm}), thus reducing the weight updates to
\begin{equation*} \Delta \mathbf{w}^{(p-1)} = -\mathcal{H}_{\mathbf{w}^{(p-1)}\mathbf{w}^{(p-1)}}^{-1} \left( \frac{\partial E}{\partial \mathbf{w}^{(p-1)}} \right)^*. \label{eq:PseudUp}\end{equation*}
Convergence using the pseudo-Newton algorithm will generally be faster than gradient descent.  The trade off for computational efficiency over Newton's method is somewhat slower convergence, though if the activation functions in the holomorphic MLP are in addition onto, the performance of the pseudo-Newton versus Newton algorithms should be similar  \cite{Kreutz-Delgado2009}.

\begin{corollary}[Pseudo-Newton Backpropagation Algorithm for Holomorphic Neural Networks]
The weight updates for the holomorphic MLP with activation functions satisfying $$\overline{g(z)}=g(\overline{z}),$$ $p=1,...,L$, using the backpropagation algorithm with the pseudo-Newton's method are given by
\begin{equation*} \Delta \mathbf{w}^{(p-1)} = -\mathcal{H}_{\mathbf{w}^{(p-1)}\mathbf{w}^{(p-1)}}^{-1} \left( \frac{\partial E}{\partial \mathbf{w}^{(p-1)}} \right)^*, \label{eq:PseudUpCor}\end{equation*}
where the entries of the Hessian matrices $\mathcal{H}_{\mathbf{w}^{(p-1)}\mathbf{w}^{(p-1)}}$ for $1\le p\le L$ are given by (\ref{eq:hessp}) in Theorem \ref{thm:complexnewtbackprop}. \label{cor:complexp-newtbackprop} \end{corollary}

\section{The One-Step Newton Steplength Algorithm for Real-Valued Complex Functions}

\par
A significant problem encountered with Newton's method and other minimization algorithms is the tendency of the iterates to ``overshoot.'' If this happens, the iterates may not decrease the function value at each step \cite{Ortega1970}.  For functions on real domains, it is known that for any minimization algorithm, careful choice of the sequence of steplengths via various steplength algorithms will guarantee a descent method.  Steplength algorithms for minimization of real-valued functions on complex domains have been discussed in the literature \cite{Ang2001,Hanna2003,Manton2002,Sorber2011}.  In \cite{Manton2002}, the problem was addressed by imposing unitary conditions on the input vectors.  In \cite{Sorber2011}, steplength algorithms were proposed for the BFGS method, which is an approximation to Newton's method.  With regard to applications in neural networks, variable steplength algorithms exist for least mean square error algorithms, and these algorithms have been adapted to the gradient descent backpropagation algorithm for fully complex-valued neural networks with analytic activation functions \cite{Ang2001,Goh2005}.  Fully adaptive gradient descent algorithms for complex-valued neural networks have also been proposed \cite{Hanna2003}.  However, these algorithms do not apply to the Newton backpropagation algorithm.

\par
To provide a steplength algorithm that guarantees convergence of Newton's method for real-valued complex functions, we need the following definitions.  Let $f:\Omega\subseteq\C^k\to\R$.  The function $f$ is called real differentiable ($\R$-differentiable) if it is (Frechet) differentiable as a mapping
$$f(\x,\y):D:=\left\{ \left( \begin{array}{cc} \x \\ \y\end{array}\right)\in\R^{2k} \left\vert \begin{array}{cc} \x,\y\in\R^k \\ \z=\x+i\y\in\Omega \end{array} \right.\right\} \R^{2k}\to\R.$$
We then define a stationary point of $f$ to be a stationary point in the sense of the function $f(\x,\y):D\subseteq\R^{2k}\to\R.$ If $f$ is twice $\R$-differentiable, let $\mathcal{H}_{\z\z}$ and $\mathcal{H}_{\overline{\z}\z}$ denote the Hessian matrices of $f$ with respect to $\z$ given by (\ref{eq:hessdef}).

\par
Let $\z(0)\in\Omega$.  If $\Omega$ is open, we define the level set of $\z(0)$ under $f$ on $\Omega$ to be
\begin{equation}
L_{\C^k}(f(\z(0)))= \left\{ \z\in\Omega \, \vert \, f(\z)\le f(\z(0)) \right\},
\label{eq:levelset}
\end{equation}
and let $L_{\C^k}^0(f(\z(0)))$ be the path-connected component of $L_{\C^k}(f(\z(0)))$ containing $\z(0)$.  To discuss rate of convergence, recall that the root-convergence factors (R-factors) of a sequence $\{\z(n)\}\subseteq\C^k$ that converges to $ \hat{\z}\in\C^k$ are
\begin{equation}
R_p\{ \z(n)\} = \left\{ \begin{array}{ll} \limsup_{n\to\infty} \Vert \z(n)- \hat{\z}\Vert^{1/n}_{\C^k} & \textrm{if }p=1, \\ \limsup_{n\to\infty} \Vert \z(n)- \hat{\z}\Vert^{1/p^n}_{\C^k} & \textrm{if } p>1, \end{array}\right.
\label{eq:Rfactors}
\end{equation}
and the sequence is said to have at least an R-linear rate of convergence if $R_1\{ \z(n)\} <1$.

The following theorem gives the one-step Newton steplength algorithm to adjust the sequence of steplengths for minimization of a real-valued complex function using Newton's method.  We provide the detailed proof in Appendix B.

\begin{theorem}[Convergence of the Complex Newton Algorithm with Complex One-Step Newton Steplengths]
Let $f:\Omega\subseteq\C^k\to\R$ be twice-continuously $\R$-differentiable on the open convex set $\Omega$ and assume that $ L^0_{\C^k}(f(\z(0)))$ is compact for $\z(0)\in\Omega$.  Suppose for all $\z\in\Omega$,
$$\mathrm{Re}\{ \mathbf{h}^* \mathcal{H}_{\z\z}(\z)\mathbf{h} + \mathbf{h}^* \mathcal{H}_{\overline{\z}\z}(\z)\overline{\mathbf{h}}\}>0 \textrm{ for all }\mathbf{h}\in\C^k.$$
Assume $f$ has a unique stationary point $\hat{\z}\in L^0_{\C^k}(f(\z(0)))$, and fix $\epsilon\in (0,1]$.  Consider the iteration 
\begin{equation}
\z(n+1)=\z(n)-\omega(n)\mu(n)\p(n), \textrm{ } n=0,1,...,
\label{eq:iteration}
\end{equation}
where the $\p(n)$ are the nonzero complex Newton updates
\begin{equation}
\begin{split}
\p(\z(n))= &-\left[\mathcal{H}_{\z\z}(\z(n))-\mathcal{H}_{\overline{\z}\z}(\z(n))\mathcal{H}_{\overline{\z}\overline{\z}}(\z(n))^{-1}\mathcal{H}_{\z\overline{\z}}(\z(n))\right]^{-1}\\
& \cdot\left[ \mathcal{H}_{\overline{\z}\z}(\z(n))\mathcal{H}_{\overline{\z}\overline{\z}}(\z(n))^{-1} \left(\frac{\partial f}{\partial \overline{\z}}(\z(n)) \right)^*-\left(\frac{\partial f}{\partial \z}(\z(n)) \right)^*\right],
\end{split}
\label{eq:Newtupdatesone-step}
\end{equation}
the steplengths $\mu(n)$ are given by
\begin{equation*}
\mu(n)=\frac{\mathrm{Re} \{ \frac{\partial f}{\partial \z} (\z(n))\p(n)\}}{\mathrm{Re}\{ \p(n)^* \mathcal{H}_{\z\z}(\z(n))\p(n) + \p(n)^*\mathcal{H}_{\overline{\z}\z}(\z(n))\overline{\p(n)}\}},
\end{equation*}\normalsize
and the underrelaxation factors $\omega(n)$ satisfy
\begin{equation}
0\le \epsilon \le \omega(n) \le \frac{2}{\gamma(n)}-\epsilon,
\label{eq:OmegaDef}
\end{equation}
where, taking $\z=\z(n)$ and $\p=\p(n)$,
\begin{equation}
\begin{split}
\gamma(n)&=\sup \left. \left\{  \frac{\mathrm{Re}\{\p^* \mathcal{H}_{\z\z}(\z-\mu\p)\p + \p^* \mathcal{H}_{\overline{\z}\z}(\z-\mu\p)\overline{\p}\}}{\mathrm{Re}\{\p^* \mathcal{H}_{\z\z}(\z)\p + \p^* \mathcal{H}_{\overline{\z}\z}(\z)\overline{\p}\}} \right. \right\vert \\
&\hspace{50mm} \left.\begin{array}{c} \mu>0, \textrm{ } f(\z-\nu\p)<f(\z)\\
\textrm{for all }\nu\in(0,\mu]\end{array}\right\}.
\end{split}
\label{eq:GammaOmegaDef}
\end{equation}
Then $\lim_{n\to\infty}\z(n)=\hat{\z},$ and the rate of convergence is at least R-linear.
\label{thm:ConvergenceNewton}
\end{theorem}

\par
To apply the one-step Newton steplength algorithm to the Newton's method or pseudo-Newton's method backpropagation algorithm for complex-valued holomorphic multilayer perceptrons, at the $n$th iteration in the training process, the one-step Newton steplength for the $p$th step in the backpropagation ($1\le p\le L$) is
\begin{equation}
\mu_p(n)=\frac{-\Re \left( \frac{\partial E}{\partial \mathbf{w}} \Delta \mathbf{w}\right)}{\Re \left\{ (\Delta\mathbf{w})^* \mathcal{H}_{\mathbf{w}\mathbf{w}}\Delta\mathbf{w} + (\Delta\mathbf{w})^* \mathcal{H}_{\overline{\mathbf{w}}\mathbf{w}}\overline{\Delta\mathbf{w}} \right\} },
\label{eq:steplengthforbackprop}
\end{equation}
where $\Delta \mathbf{w}=\Delta\mathbf{w}^{(p-1)}$ is the weight update for the $p$th layer of the network given by Theorem \ref{thm:complexnewtbackprop} or Corollary \ref{cor:complexp-newtbackprop}, respectively, and $\mathbf{w}=\mathbf{w}^{(p-1)}$.  (Recall (\ref{eq:updatestep}), so that here $\p(n)=-\Delta\mathbf{w}^{(p-1)}$ in (\ref{eq:iteration}).) For the pseudo-Newton's method backpropagation, we set $\mathcal{H}_{\overline{\mathbf{w}^{(p-1)}}\mathbf{w}^{(p-1)}}=\mathcal{H}_{\mathbf{w}^{(p-1)}\overline{\mathbf{w}^{(p-1)}}}=0$ in (\ref{eq:Newtupdatesone-step}) to obtain the pseudo-Newton updates $\Delta \mathbf{w}^{(p-1)}$ given in Corollary \ref{cor:complexp-newtbackprop}, but leave $\mathcal{H}_{\overline{\mathbf{w}^{(p-1)}}\mathbf{w}^{(p-1)}}$ as calculated in Theorem \ref{thm:complexnewtbackprop} in (\ref{eq:steplengthforbackprop}).  In theory, for the $n$th iteration in the training process, we should choose the underrelaxation factor $\omega_p(n)$ for the $p$th step in the backpropagation $(1\le p\le L)$ according to (\ref{eq:OmegaDef}) and (\ref{eq:GammaOmegaDef}).  However, in practical application it suffices to take the underrelaxation factors to be constant and they may be chosen experimentally to yield convergence of the error function (see our results in Section VII).  It is also not necessary in practical application to verify all the conditions of Theorem \ref{thm:ConvergenceNewton}.  In particular we may assume that the error function has a stationary point sufficiently close to the initial weights since the initial weights were chosen specifically to be ``nearby'' a stationary point, and that the stationary point is unique in the appropriate compact level set of the initial weights since the set of zeros of the error function has measure zero.

\section{Experiments}

\begin{table}[b] \begin{center}
\begin{tabular}{|c|c|}\hline
Input Pattern & Output\\ \hline
0 0 & 0\\
1 0 & 1\\
0 1 & 1\\
1 1 & 0\\ \hline
\end{tabular}
\vspace{2mm}
\caption{XOR Training Set}
\label{tab:XORdata}\end{center}
\end{table}

\par
To test the efficiency of the algorithms in the previous sections, we will compare the results of applying the gradient descent method, Newton's method, and the pseudo-Newton's method to a holomorphic MLP trained with data from the real-valued exclusive-or (XOR) problem (see Table \ref{tab:XORdata}).  Note that the complex-valued XOR problem has different criteria for the data set \cite{Savitha2009}.  We use the real-valued XOR problem as we desire a complex-valued network to process real as well as complex data.

\par
The XOR problem is frequently encountered in the literature as a test case for backpropagation algorithms \cite{Pande2007}.  A multilayer network is required to solve it: without hidden units the network is unable to distinguish overlapping input patterns which map to different output patterns, e.g. $(0,0)$ and $(1,0)$ \cite{Rumelhart1986}.  We use a two-layer network with $m=2$ input nodes, $K=4$ hidden nodes, and $C=1$ output nodes.  Any Boolean function of $m$ variables can be trained to a two-layered real-valued neural network with $2^m$ hidden units.  Modeling after the real case we choose $K=2^m$, although this could perhaps be accomplished with fewer hidden units, as $2^{m-1}$ is a smaller upper bound for real-valued neural networks \cite{Hassoun1995}.  Some discussion of approximating Boolean functions, including the XOR and parity problems, using complex-valued neural networks is given in \cite{Nemoto1992}.

\par
In our experiments, the activation functions are taken to be the same for both the hidden and output layers of the network.  The activation function is either the sigmoidal function or its third degree\footnote{One can take a higher degree Taylor polynomial approximation, but this is sufficient for our purposes.} Taylor polynomial approximation 
$$g(z)=\frac{1}{1+\exp(-z)} \textrm{ or } T(z)=\frac{1}{2} +\frac{1}{4}z-\frac{1}{48}z^3.$$
Notice that while $g(z)$ has poles near zero, the polynomial $T(z)$ is analytic on the entire complex plane and bounded on bounded regions (see Figure \ref{fig:graphsigvspoly}).

\begin{figure}[t]\begin{center}
\includegraphics[scale=0.2]{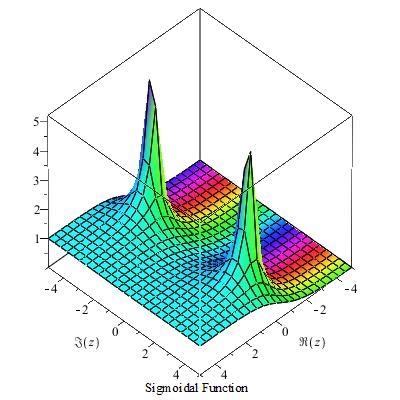} \includegraphics[scale=0.2]{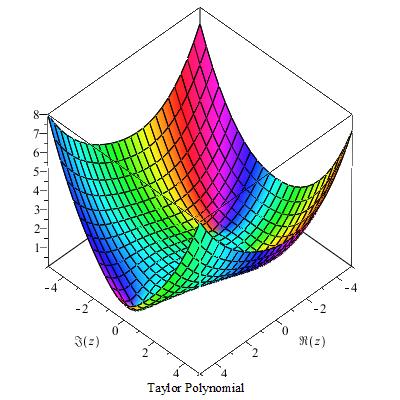}
\caption{The sigmoidal function (left) has two poles in a region near $0$, while a Taylor polynomial approximation (right) of the sigmoidal function is bounded on the same region.}
\label{fig:graphsigvspoly}\end{center}
\end{figure}

\par
For each activation function we trained the network using the gradient descent backpropagation algorithm, the Newton backpropagation algorithm, and the pseudo-Newton backpropagation algorithm.  The real and imaginary parts of the initial weights for each trial were chosen randomly from the interval $[-1,1]$ according to a uniform distribution.  In each case the network was trained to within $0.001$ error.  One hundred trials were performed for each activation function and each backpropagation algorithm (note that the same set of random initial weights was used for each set of trials).  For the trials using the gradient descent backpropagation algorithm, a constant learning rate ($\mu$) was used.  It is known that for the gradient descent algorithm for real-valued neural networks, some learning rates will result in nonconvergence of the error function \cite{LeCun1991}.  There is experimental evidence that for elementary transcendental activation functions used in complex-valued neural networks, sensitivity of the gradient descent algorithm to the choice of the learning rate can result in nonconvergence of the error function as well, and this is not necessarily affected by changes in the initial weight distribution \cite{Savitha2009}.  To avoid these problems, a learning rate of $\mu=1$ was chosen both to guarantee convergence and to yield fast convergence (as compared to other values of $\mu$).  For the trials using the Newton and pseudo-Newton backpropagation algorithms, a variable learning rate (steplength) was chosen according to the one-step Newton steplength algorithm (Theorem \ref{thm:ConvergenceNewton}) to control the problem of ``overshooting'' of the iterates and nonconvergence of the error function when a fixed learning rate was used.  For both the Newton and pseudo-Newton trials, a constant underrelaxation factor of $\omega=0.5$ was used; this was chosen to yield the best chance for convergence of the error function. The results are summarized in Table \ref{tab:experResults}.

\begin{table}[!t]\begin{center}\tiny
\begin{tabular}{|c|c|c|c|c|c|}\hline
                &               &                            &                        & \bf{Number of}   & \bf{Average}     \\
\bf{Activation} & \bf{Training} & \bf{Learning}              & \bf{Underrelaxation}   & \bf{Successful}      & \bf{Number of}   \\
\bf{Function}   & \bf{Method}   & \bf{Rate ($\mathbf{\mu}$)} & \bf{Factor ($\omega$)} & \bf{Trials}          & \bf{Iterations*} \\ \hline

Sigmoidal       & Gradient      & $\mu=1$                    & None                   & 93                   & 1258.9           \\
                & Descent       &                            &                        &                      &                  \\ \hline

Sigmoidal       & Newton        & One-Step                   & $\omega=0.5$           & 5                    & 7.0              \\
                &               & Newton                     &                        &                      &                  \\ \hline
      
Sigmoidal       & Pseudo-       & One-Step                   & $\omega=0.5$           & 78                   & 7.0              \\
                & Newton        & Newton                     &                        &                      &                  \\ \hline
                
Polynomial      & Gradient      & $\mu=1$                    & None                   & 93                   & 932.2            \\
                & Descent       &                            &                        &                      &                  \\ \hline

Polynomial      & Newton        & One-Step                   & $\omega=0.5$           & 53                   & 107.9            \\
                &               & Newton                     &                        &                      &                  \\ \hline
      
Polynomial      & Pseudo-       & One-Step                   & $\omega=0.5$           & 99                   & 23.7             \\
                & Newton        & Newton                     &                        &                      &                  \\ \hline     \end{tabular}
\vspace{2mm}

\tiny{*Over the successful trials.}
\caption{XOR Experiment Results}
\label{tab:experResults}\end{center}
\end{table}
\normalsize

\begin{table}[!t]\begin{center}\tiny
\begin{tabular}{|c|c|c|c|c|c|c|}\hline
                &               &              &           & \bf{Undefined} &               & \bf{Total}        \\
\bf{Activation} & \bf{Training} & \bf{Local}   & \bf{Blow} & \bf{Floating}  & \bf{Singular} & \bf{Unsuccessful} \\
\bf{Function}   & \bf{Method}   & \bf{Minimum} & \bf{Up}   & \bf{Point}     & \bf{Matrix}   & \bf{Trials}       \\ \hline

Sigmoidal       & Gradient      & 1            & 0         & 6              & N/A           & 7                 \\
                & Descent       &              &           &                &               &                   \\ \hline

Sigmoidal       & Newton        & 0            & 0         & 68             & 27            & 95                \\ \hline

Sigmoidal       & Pseudo-       & 0            & 0         & 14             & 8             & 22                \\
                & Newton        &              &           &                &               &                   \\ \hline
                
Polynomial      & Gradient      & 0            & 0         & 7              & N/A           & 7                 \\
                & Descent       &              &           &                &               &                   \\ \hline

Polynomial      & Newton        & 26           & 2         & 2              & 17            & 47                \\ \hline
      
Polynomial      & Pseudo-       & 1            & 0         & 0              & 0             & 1                 \\
                & Newton        &              &           &                &               &                   \\ \hline
\end{tabular}
\vspace{2mm}
\caption{Unsuccessful Trials}
\label{tab:unsuccess}\end{center}
\end{table}

\normalsize

\par
Over the successful trials, the polynomial activation function performed just as well as the traditional sigmoidal function for the gradient descent backpropagation algorithm and yielded more successful trials than the sigmoidal function for the Newton and pseudo-Newton backpropagation algorithms.  We define a successful trial to be one in which the error function dropped below $0.001$.  We logged four different types of unsuccessful trials (see Table \ref{tab:unsuccess}).  Convergence of the error function to a local minimum occurred when, after at least 50,000 iterations for gradient descent and 5,000 iterations for the Newton and pseudo-Newton algorithms, the error function remained above $0.001$ but had stabilized to within $10^{-10}$ between successive iterations.  This occurred more frequently in the Newton's method trails than the gradient descent trials, which was expected due to the known sensitivity of Newton's method to the initial points.  A blow up of the error function occurred when, after the same minimum number of iterations as above, the error function had increased to above $10^{10}$.  The final value of the error function was sometimes an undefined floating point number, probably the result of division by zero.  This occurred less frequently with the polynomial activation function than with the sigmoidal activation function.  Finally, the last type of unsuccessful trial resulted from a singular Hessian matrix (occurring only in the Newton and pseudo-Newton trials).  This, necessarily, halted the backpropagation process, and occurred less frequently with the polynomial activation function than with the sigmoidal activation function.

\par
As for efficiency, the Newton and pseudo-Newton algorithms required significantly fewer iterations of the backpropagation algorithm to train the network than the gradient descent method for each activation function.  In addition to producing fewer unsuccessful trials, the pseudo-Newton algorithm yielded a lower average number of iterations than the Newton algorithm for the polynomial activation function and the same average number of iterations as the Newton algorithm for the sigmoidal activation function.  The network with polynomial activation function trained using the pseudo-Newton algorithm produced the fewest unsuccessful trials.  Overall, we conclude that the use of the polynomial activation function yields more consistent convergence of the error function than the use of the sigmoidal activation function, and the use of the Newton and pseudo-Newton algorithms yields significantly fewer training iterations than the use of the gradient descent method.

\section{Conclusion}

\par
We have developed the backpropagation algorithm using Newton's method for complex-valued holomorphic multilayer perceptrons. The extension of real-valued neural networks to complex-valued neural networks is natural and doing so allows the proper treatment of the phase information. However, the choice of nonlinear activation functions poses a challenge in the backpropagation algorithm. The usual complex counterparts of the commonly used real-valued activation functions are no longer unbounded: they have poles near zero, while other choices are not fully complex-valued functions. To provide experimental evidence for the choice of holomophic functions as activation functions in addition to mathematical reasoning, we compared the results of using the complex-valued sigmoidal function as activation functions and the results of using its Taylor polynomial approximation as activation functions. Our experiments showed that when Newton's method was used for the XOR example, Taylor polynomial approximations are better choices. The use of polynomials as activation functions allows the possibility of rigorous analysis of performance of the algorithm, as well as making connections with other topics of complex analysis, which are virtually nonexistent in complex-valued neural network studies so far.  These topics are under investigation currently.

\bibliographystyle{unsrt}

\bibliography{references}

\begin{thebibliography}{10}

\bibitem{Hirose2012}
Akira Hirose.
\newblock {\em Complex-Valued Neural Networks}, volume 400 of {\em Studies in
  Computational Intelligence}.
\newblock Springer-Verlag Berlin Heidelberg, New York, 2nd edition, 2012.

\bibitem{Hirose2011}
Akira Hirose.
\newblock Nature of complex number and complex-valued neural networks.
\newblock {\em Frontiers of Electrical and Electronic Engineering in China},
  6(1):171--180, 2011.

\bibitem{Mukherjee2012}
Indrajit Mukherjee and Srikanta Routroy.
\newblock Comparing the performance of neural networks developed by using
  levenberg–marquardt and quasi-newton with the gradient descent algorithm
  for modelling a multiple response grinding process.
\newblock {\em Expert Systems with Applications}, 39:2397--2407, February 2012.

\bibitem{Conway1978}
John~B. Conway.
\newblock {\em Functions of One Complex Variable I}.
\newblock Graduate Texts in Mathematics. Springer Science+Business Media, Inc.,
  New York, 2 edition, 1978.

\bibitem{Georgiou1992}
George~M. Georgiou and Cris Koutsougeras.
\newblock Complex domain backpropagation.
\newblock {\em IEEE Transactions on Circuits and Systems-II: Analog and Digital
  Signal Processing}, 39(5):300--334, May 1992.

\bibitem{Jalab2011}
Hamid~A. Jalab and Rabha~W. Ibrahim.
\newblock New activation functions for complex-valued neural network.
\newblock {\em International Journal of the Physical Sciences},
  6(7):1766--1772, April 2011.

\bibitem{Kim2001}
Taehwan Kim and T{\"u}lay Adali.
\newblock Approximation by fully complex mlp using elementary transcendental
  functions.
\newblock In {\em Neural Networks for Signal Processing XI, 2001. Proceedings
  of the 2001 IEEE Signal Processing Society Workshop}, pages 203--212. IEEE,
  2001.

\bibitem{Pande2007}
Anupama Pande and Vishik Goel.
\newblock Complex-valued neural network in image recognition: A study on the
  effectiveness of radial basis function.
\newblock {\em World Academy of Science, Engineering and Technology},
  26:220--225, 2007.

\bibitem{Amin2009-2}
Md.~Faijul Amin and Kazuyuki Murase.
\newblock Single-layered complex-valued neural network for real-valued
  classification problems.
\newblock {\em Neurocomputing}, 72:945--955, 2009.

\bibitem{Amin2009}
Md.~Faijul Amin, Md.~Monirul Islam, and Kazuyuki Murase.
\newblock Ensemble of single-layered complex-valued neural networks for
  classification tasks.
\newblock {\em Neurocomputing}, 72:2227--2234, 2009.

\bibitem{Amin2008}
Md.~Faijul Amin, Md.~Monirul Islam, and Kazuyuki Murase.
\newblock Single-layered complex-valued neural networks and their ensembles for
  real-valued classification problems.
\newblock In {\em 2008 International Joint Conference on Neural Networks},
  pages 2500--2506. IEEE, 2008.

\bibitem{Hanna2003}
Andrew~Ian Hanna and Danilo~P. Mandic.
\newblock A fully adaptive normalized nonlinear gradient descent algorithm for
  complex-valued nonlinear adaptive filters.
\newblock {\em IEEE Transactions on Signal Processing}, 51(10):2540--2549,
  October 2003.

\bibitem{Kim2002}
Taehwan Kim and T{\"u}lay Adali.
\newblock Fully complex multi-layer perceptron network for nonlinear signal
  processing.
\newblock {\em Journal of VLSI Signal Processing Systems}, 32(1/2):29--43,
  August-September 2002.

\bibitem{Burse2011}
Kavita Burse, Anjana Pandey, and Ajay Somkuwar.
\newblock Convergence analysis of complex valued multiplicative neural network
  for various activation functions.
\newblock In {\em 2011 International Conference on Computational Intelligence
  and Communication Systems}, pages 279--282, 2011.

\bibitem{Li2005}
Ming-Bin Li, Guang-Bin Huang, P.~Saratchandran, and N.~Sundararajan.
\newblock Fully complex extreme learning machine.
\newblock {\em Neurocomputing}, 68:306--314, October 2005.

\bibitem{Savitha2011}
R.~Savitha, S.~Suresh, N.~Sundararajan, and H.J. Kim.
\newblock Fast learning fully complex-valued classifiers for real-valued
  classification problems.
\newblock In D.~Liu et~al, editor, {\em Advances in Neural Networks--ISNN 2011,
  Part I}, volume 6675 of {\em Lecture Notes in Computer Science}, pages
  602--609. Springer-Verlag Berlin Heidelberg, 2011.

\bibitem{Leung1991}
Henry Leung and Simon Haykin.
\newblock The complex backpropagation algorithm.
\newblock {\em IEEE Transactions on Signal Processing}, 39(9):2101--2104,
  September 1991.

\bibitem{Savitha2009}
R.~Savitha, S.~Suresh, N.~Sundararajan, and P.~Saratchandran.
\newblock A new learning algorithm with logarithmic performance index for
  complex-valued neural networks.
\newblock {\em Neurocomputing}, 72:3771--3781, 2009.

\bibitem{Amin2011}
Md.~Faijul Amin, Ramasamy Savitha, Muhammad~Ilias Amin, and Kazuyuki Murase.
\newblock Complex-valued functional link network design by orthogonal least
  squares method for function approximation problems.
\newblock In {\em Proceedings of the International Joint Conference on Neural
  Networks}, pages 1489--1496, July/August 2011.

\bibitem{Amin2012}
Md.~Faijul Amin, Ramasamy Savitha, Muhammad~Ilias Amin, and Kazuyuki Murase.
\newblock Orthogonal least squares based complex-valued functional link
  network.
\newblock {\em Neural Networks}, 32:257--266, 2012.
\newblock 2012 Special Issue.

\bibitem{Reed1980}
Michael Reed and Barry Simon.
\newblock {\em Methods of Modern Mathematical Physics I: Functional Analysis}.
\newblock Academic Press, London, 1980.

\bibitem{Buchholz2008}
Sven Buchholz and Gerald Sommer.
\newblock On clifford neurons and clifford multi-layer perceptrons.
\newblock {\em Neural Networks}, 21:925--935, 2008.

\bibitem{Al-Haik2003}
M.S. Al-Haik, H.~Garmestani, and I.M. Navon.
\newblock Truncated-newton training algorithm for neurocomputational
  viscoplastic model.
\newblock {\em Computational methods in applied mechanics and engineering},
  192:2249--2267, 2003.

\bibitem{Beigi1993}
H.S.M. Beigi and C.J. Li.
\newblock Learning algorithms for neural networks based on quasi-newton methods
  with self-scaling.
\newblock {\em Journal of Dynamical Systems, Measurement, and Control},
  115:38--43, March 1993.

\bibitem{Hagan1994}
Martin~T. Hagan and Mohammad~B. Menhaj.
\newblock Training feedforward networks with the marquardt algorithm.
\newblock {\em IEEE Transactions on Neural Networks}, 5(6):989--993, November
  1994.

\bibitem{Yu2011}
Hao Yu and Bogdan~M. Wilamowski.
\newblock Levenberg-marquardt training.
\newblock In {\em Industrial Electronics Handbook, Vol. 5: Intelligent
  Systems}, chapter~12, pages 12--1 -- 12--15. CRC Press, 2 edition, 2011.

\bibitem{Goh2005}
Su~Lee Goh and Danilo~P. Mandic.
\newblock A class of gradient-adaptive step size algorithms for complex-valued
  nonlinear neural adaptive filters.
\newblock In {\em IEEE International Conference on Acoustics, Speech, and
  Signal Processing, 2005. Proceedings. (ICASSP '05)}, volume~5, pages
  V/253--V/256. IEEE, May 2005.

\bibitem{Zimmermann2011}
Hans~Georg Zimmermann, Alexey Minin, and Victoria Kusherbaeva.
\newblock Comparison of the complex valued and ral valued neural networks
  trained with gradient descent and random search algorithms.
\newblock In {\em European Symposium on Artificial Neural Networks,
  Computational Intelligence and Machine Learning}, pages 213--218, Bruges
  (Belgium), April 2011.

\bibitem{Li2008}
Hualiang Li and T{\"u}lay Adali.
\newblock Complex-valued adaptive signal processing using nonlinear functions.
\newblock {\em EURASIP Journal on Advances in Signal Processing}, 2008, 2008.

\bibitem{Kreutz-Delgado2009}
Ken Kreutz-Delgado.
\newblock The complex gradient operator and the {$\C\R$}-calculus.
\newblock University of California, San Diego, Version UCSD-ECE275CG-S2009v1.0,
  25 June 2009. arXiv:0906.4835v1 [math.OC], June 2009.

\bibitem{Ortega1970}
J.M. Ortega and W.C. Rheinboldt.
\newblock {\em Iterative Solution of Nonlinear Equations in Several Variables}.
\newblock Academic Press, Inc., New York, NY, 1970.

\bibitem{Ang2001}
Wee-Peng Ang and B.~Farhang-Boroujeny.
\newblock A new class of gradient adaptive step-size lms algorithms.
\newblock {\em IEEE Transactions on Signal Processing}, 49(4):805--810, April
  2001.

\bibitem{Manton2002}
Jonathan~H. Manton.
\newblock Optimization algorithms exploiting unitary constraints.
\newblock {\em IEEE Transactions on Signal Processing}, 50(3):635--650, March
  2002.

\bibitem{Sorber2011}
Laurent Sorber, Marc~Van Barel, and Lieven~De Lathauwer.
\newblock Unconstrained optimization of real functions in complex variables.
\newblock Technical Report TW592, Katholieke Universiteit Leuven, Heverlee
  (Belgium), April 2011.

\bibitem{Rumelhart1986}
D.E. Rumelhart, G.E. Hinton, and R.J. Williams.
\newblock Learning internal representations by error propagation.
\newblock In D.E. Rumelhart and J.L. McCelland, editors, {\em Parallel
  Distributed Processing: Explorations in the Microstructure of Cognition},
  volume~1, chapter~8. Foundations M.I.T. Press, Cambridge, MA, 1986.

\bibitem{Hassoun1995}
Mohamad~H. Hassoun.
\newblock {\em Fundamentals of Artificial Neural Networks}.
\newblock The MIT Press, Cambridge, MA, 1995.

\bibitem{Nemoto1992}
Iku Nemoto and Tomoshi Kono.
\newblock Complex neural networks.
\newblock {\em Systems and Computers in Japan}, 23(8):75--84, 1992.
\newblock Translated from Denshi Joho Tsushin Gakkai Ronbunshi, Vol. 74-D-II,
  No. 9, pp. 1282-1288, September 1991.

\bibitem{LeCun1991}
Yann~Le Cun, Ido Kanter, and Sara~A. Solla.
\newblock Eigenvalues of covariance matrices: Application to neural-network
  learning.
\newblock {\em Physical Review Letters}, 66(18):2396--2399, May 1991.

\bibitem{Remmert1991}
Reinhold Remmert.
\newblock {\em Theory of Complex Functions}.
\newblock Springer-Verlag, New York, NY, 1991.

\end{thebibliography}

\appendix

\section{Derivation of the Entries of the Hessian Matrices for the Newton's Method Backpropagation Algorithm}

We give the detail, which was omitted in the main body of the paper, for the computation of the entries of the Hessian matrices $\mathcal{H}_{\overline{\mathbf{w}}\mathbf{w}}$ for the $(p-1)$th layer of the holomorphic multilayer perceptron, $1\le p\le L$, which are given recursively by (\ref{eq:hidconjdeltas}) and (\ref{eq:deltapbyxconj}), in a manner similar to the computation of the Hessian matrices $\mathcal{H}_{\mathbf{w}\mathbf{w}}$ given in Section IV.  Using the cogradients (\ref{eq:grad}) we have:
\begin{equation}
\frac{\partial}{\partial \overline{w^{(p-1)}_{ba}}} \left(\frac{\partial E}{\partial w^{(p-1)}_{ji}} \right)^* 
= \frac{1}{N} \sum_{t=1}^N \frac{\partial E^{(p)}_{tj}}{\partial \overline{w^{(p-1)}_{ba}}}  \overline{x^{(p-1)}_{ti}},
\label{eq:hidconjcomp}
\end{equation}
where $j,b=1,...,K_p$ and $i,a=1,...,K_{p-1}$.  Using (\ref{eq:deltap}),
\begin{equation*}
\begin{split}
&\frac{\partial E^{(p)}_{tj}}{\partial \overline{w^{(p-1)}_{ba}}} = \frac{\partial}{\partial \overline{w^{(p-1)}_{ba}}} \left[\left(\sum_{\eta=1}^{K_{p+1}} E^{(p+1)}_{t\eta} \overline{w^{(p)}_{\eta j}}\right) g_p'\left(\overline{\left(x^{(p)}_{tj}\right)^{\textrm{net}}}\right)\right]\\
&=g_p'\left(\overline{\left(x^{(p)}_{tj}\right)^{\textrm{net}}}\right) \sum_{\eta=1}^{K_{p+1}} \frac{\partial E^{(p+1)}_{t\eta}}{\partial \overline{w^{(p-1)}_{ba}}}  \overline{w^{(p)}_{\eta j}}+ \frac{\partial g_p'\left(\overline{\left(x^{(p)}_{tj}\right)^{\textrm{net}}}\right)}{\partial \overline{w^{(p-1)}_{ba}}}\sum_{\eta=1}^{K_{p+1}} E^{(p+1)}_{t\eta} \overline{w^{(p)}_{\eta j}},\\
\end{split}
\end{equation*}
where
\begin{equation*}
\begin{split}
\frac{\partial E^{(p+1)}_{t\eta}}{\partial \overline{w^{(p-1)}_{ba}}} &= \frac{\partial E^{(p+1)}_{t\eta}}{\partial x^{(p)}_{tb}} \frac{\partial x^{(p)}_{tb}}{\partial \overline{w^{(p-1)}_{ba}}} +\frac{\partial E^{(p+1)}_{t\eta}}{\partial \overline{x^{(p)}_{tb}}} \frac{\overline{\partial x^{(p)}_{tb}}}{\partial \overline{w^{(p-1)}_{ba}}}\\
&= \frac{\partial E^{(p+1)}_{t\eta}}{\partial \overline{x^{(p)}_{tb}}} \left[ \frac{\partial \overline{x^{(p)}_{tb}}}{\partial \overline{\left(x^{(p)}_{tb}\right)^{\textrm{net}}}} \frac{\partial \overline{\left(x^{(p)}_{tb}\right)^{\textrm{net}}}}{\partial  \overline{w^{(p-1)}_{ba}}} + \frac{\partial \overline{x^{(p)}_{tb}}}{\partial \left(x^{(p)}_{tb}\right)^{\textrm{net}}} \frac{\partial \left(x^{(p)}_{tb}\right)^{\textrm{net}}}{\partial  \overline{w^{(p-1)}_{ba}}}\right]\\
&= \frac{\partial E^{(p+1)}_{t\eta}}{\partial \overline{x^{(p)}_{tb}}} g_p'\left(\overline{\left(x^{(p)}_{tb}\right)^{\textrm{net}}}\right) \overline{x^{(p-1)}_{ta}}
\end{split}
\label{eq:halfcrossdelta1}
\end{equation*}
and
\begin{equation*}
\begin{split}
&\frac{\partial g_p'\left(\overline{\left(x^{(p)}_{tj}\right)^{\textrm{net}}}\right)}{\partial \overline{w^{(p-1)}_{ba}}}\\
&\hspace{10mm}= \frac{\partial g_p'\left(\overline{\left(x^{(p)}_{tj}\right)^{\textrm{net}}}\right)}{\partial \overline{\left(x^{(p)}_{tj}\right)^{\textrm{net}}}} \frac{\partial \overline{\left(x^{(p)}_{tj}\right)^{\textrm{net}}}}{\partial \overline{w^{(p-1)}_{ba}}}+ \frac{\partial g_p'\left(\overline{\left(x^{(p)}_{tj}\right)^{\textrm{net}}}\right)}{\partial \left(x^{(p)}_{tj}\right)^{\textrm{net}}} \frac{\left(x^{(p)}_{tj}\right)^{\textrm{net}}}{\partial \overline{w^{(p-1)}_{ba}}}\\
&\hspace{10mm}=\left\{ \begin{array}{ll} g_p''\left(\overline{\left(x^{(p)}_{tj}\right)^{\textrm{net}}}\right)\overline{x_{ta}^{(p-1)}} & \textrm{if } j=b,\\ 0 & \textrm{if } j\neq b, \end{array}\right.
\end{split}
\label{eq:halfcrossdelta2}
\end{equation*}
so that
\begin{equation}
\frac{\partial E^{(p)}_{tj}}{\partial \overline{w^{(p-1)}_{ba}}} =\left\{ \begin{array}{l}

\left\{\left[ \sum_{\eta=1}^{K_{p+1}} \frac{\partial E^{(p+1)}_{t\eta}}{\partial \overline{x^{(p)}_{tb}}} \overline{w^{(p)}_{\eta j}}\right] g_p'\left(\overline{\left(x^{(p)}_{tj}\right)^{\textrm{net}}}\right)g_p'\left(\overline{\left(x^{(p)}_{tb}\right)^{\textrm{net}}}\right) \right. \\

\hspace{5mm} \left. + \left[ \sum_{\eta=1}^{K_{p+1}} E^{(p+1)}_{t\eta} \overline{w^{(p)}_{\eta j}} \right]g_p''\left( \overline{\left( x^{(p)}_{tj}\right)^{\textrm{net}}}\right)\right\} \overline{x^{(p-1)}_{ta}}\\

\hspace{55mm}\textrm{if }j=b, \\

\left[ \sum_{\eta=1}^{K_{p+1}} \frac{\partial E^{(p+1)}_{t\eta}}{\partial \overline{x^{(p)}_{tb}}} \overline{w^{(p)}_{\eta j}}\right] g_p'\left(\overline{\left(x^{(p)}_{tj}\right)^{\textrm{net}}}\right)g_p'\left(\overline{\left(x^{(p)}_{tb}\right)^{\textrm{net}}}\right)\overline{x^{(p-1)}_{ta}}\\

\hspace{55mm}\textrm{if }j\neq b. \\
 \end{array}\right.
\label{eq:deltawconj}
\end{equation}
Combining (\ref{eq:hidconjcomp}) and (\ref{eq:deltawconj}), we get
\begin{equation}
\begin{split}
&\frac{\partial}{\partial \overline{w^{(p-1)}_{ba}}} \left( \frac{\partial E}{\partial w^{(p-1)}_{ji}} \right)^*\\
&=\left\{ \begin{array}{l} 
\frac{1}{N} \sum_{t=1}^N \left\{ \left[ \sum_{\eta=1}^{K_{p+1}}\frac{\partial E^{(p+1)}_{t\eta}}{\partial \overline{x^{(p)}_{tb}}} \overline{w^{(p)}_{\eta j}} \right] g_p'( \overline{(x^{(p)}_{tj})^{\textrm{net}}})g_p'( \overline{(x^{(p)}_{tb})^{\textrm{net}}}) \right.\\
\hspace{10mm} \left. +\left[\sum_{\eta=1}^{K_{p+1}} E_{t\eta}^{(p+1)} \overline{w^{(p)}_{\eta j}} \right] g_p''( \overline{(x^{(p)}_{tj})^{\textrm{net}}}) \right\} \overline{x^{(p-1)}_{ti}}\overline{x^{(p-1)}_{ta}} \textrm{ if }j=b,\\
\frac{1}{N} \sum_{t=1}^N \left\{ \left[ \sum_{\eta=1}^{K_{p+1}}\frac{\partial E^{(p+1)}_{t\eta}}{\partial \overline{x^{(p)}_{tb}}} \overline{w^{(p)}_{\eta j}} \right] g_p'( \overline{(x^{(p)}_{tj})^{\textrm{net}}})g_p'( \overline{(x^{(p)}_{tb})^{\textrm{net}}}) \right\}\\
\hspace{45mm}\cdot \overline{x^{(p-1)}_{ti}}\overline{x^{(p-1)}_{ta}}\hspace{8mm}\textrm{ if }j\neq b,
\end{array}\right.
\end{split}
\label{eq:hidconjdeltas2}
\end{equation}
where $\frac{\partial E^{(p+1)}_{t\eta}}{\partial \overline{x^{(p)}_{tb}}}$ can be computed recursively:
\begin{equation}
\begin{split}
\frac{\partial E^{(p+1)}_{t\eta}}{\partial \overline{x^{(p)}_{tb}}} &= \frac{\partial}{\partial \overline{x^{(p)}_{tb}}} \left[ \left( \sum_{\alpha=1}^{K_{p+2}} E^{(p+2)}_{t\alpha} \overline{w^{(p+1)}_{\alpha \eta}}\right) g'_{p+1}\left(\overline{\left( x^{(p+1)}_{t\eta}\right)^{\textrm{net}}}\right) \right]\\&=g'_{p+1}\left(\overline{\left( x^{(p+1)}_{t\eta}\right)^{\textrm{net}}}\right) \sum_{\alpha=1}^{K_{p+2}} \frac{\partial E^{(p+2)}_{t\alpha}}{\partial \overline{x^{(p)}_{tb}}} \overline{w^{(p+1)}_{\alpha \eta}}\\
&\hspace{20mm} + \frac{\partial g'_{p+1}\left(\overline{\left( x^{(p+1)}_{t\eta}\right)^{\textrm{net}}}\right)}{\partial \overline{x^{(p)}_{tb}}}\sum_{\alpha=1}^{K_{p+2}} E^{(p+2)}_{t\alpha} \overline{w^{(p+1)}_{\alpha \eta}}\\
\end{split}
\label{eq:deltapbyxconj2}
\end{equation}
\begin{equation*}
\begin{split}
&=g'_{p+1}\left(\overline{\left( x^{(p+1)}_{t\eta}\right)^{\textrm{net}}}\right) \sum_{\alpha=1}^{K_{p+2}} \sum_{\beta=1}^{K_{p+1}} \left[ \frac{\partial E^{(p+2)}_{t\alpha}}{\partial x^{(p+1)}_{t\beta}} \frac{\partial x^{(p+1)}_{t\beta}}{\partial \overline{x^{(p)}_{tb}}} \right.\\
&\hspace{45mm} \left. + \frac{\partial E^{(p+2)}_{t\alpha}}{\partial \overline{x^{(p+1)}_{t\beta}}} \frac{\partial \overline{x^{(p+1)}_{t\beta}}}{\partial \overline{x^{(p)}_{tb}}}\right] \overline{w^{(p+1)}_{\alpha \eta}}\\
&\hspace{10mm} + \left[ \frac{\partial g'_{p+1}\left(\overline{\left( x^{(p+1)}_{t\eta}\right)^{\textrm{net}}}\right)}{\partial \overline{\left( x^{(p+1)}_{t\eta}\right)^{\textrm{net}}}} \frac{\partial \overline{\left( x^{(p+1)}_{t\eta}\right)^{\textrm{net}}}}{\partial \overline{x^{(p)}_{tb}}}  \right.\\
&\hspace{15mm} +\left. \frac{\partial g'_{p+1}\left(\overline{\left( x^{(p+1)}_{t\eta}\right)^{\textrm{net}}}\right)}{\partial \left( x^{(p+1)}_{t\eta}\right)^{\textrm{net}}} \frac{\partial \left( x^{(p+1)}_{t\eta}\right)^{\textrm{net}}}{\partial \overline{x^{(p)}_{tb}}}\right] \sum_{\alpha=1}^{K_{p+2}} E^{(p+2)}_{t\alpha} \overline{w^{(p+1)}_{\alpha \eta}}\\
&= g'_{p+1}\left(\overline{\left( x^{(p+1)}_{t\eta}\right)^{\textrm{net}}}\right) \sum_{\alpha=1}^{K_{p+2}} \sum_{\beta=1}^{K_{p+1}} \frac{\partial E^{(p+2)}_{t\alpha}}{\partial \overline{x^{(p+1)}_{t\beta}}} \left[ \frac{\partial \overline{x^{(p+1)}_{t\beta}}}{\partial \left(x^{(p+1)}_{t\beta}\right)^{\textrm{net}}} \frac{\partial \left(x^{(p+1)}_{t\beta}\right)^{\textrm{net}}}{\partial \overline{x^{(p)}_{tb}}}  \right.\\
&\hspace{40mm}+\left. \frac{\partial \overline{x^{(p+1)}_{t\beta}}}{\partial \overline{\left(x^{(p+1)}_{t\beta}\right)^{\textrm{net}}}} \frac{\partial \overline{\left(x^{(p+1)}_{t\beta}\right)^{\textrm{net}}}}{\partial \overline{x^{(p)}_{tb}}}\right] \overline{w^{(p+1)}_{\alpha \eta}}\\
& \hspace{10mm} +g''_{p+1}\left(\overline{\left( x^{(p+1)}_{t\eta}\right)^{\textrm{net}}}\right) \overline{w^{(p)}_{\eta b}} \sum_{\alpha=1}^{K_{p+2}} E^{(p+2)}_{t\alpha} \overline{w^{(p+1)}_{\alpha \eta}}\\
&=g'_{p+1}\left(\overline{\left( x^{(p+1)}_{t\eta}\right)^{\textrm{net}}}\right) \sum_{\alpha=1}^{K_{p+2}} \sum_{\beta=1}^{K_{p+1}} \frac{\partial E^{(p+2)}_{t\alpha}}{\partial \overline{x^{(p+1)}_{t\beta}}} \\
&\hspace{20mm} \cdot g'_{p+1}\left(\overline{\left( x^{(p+1)}_{t\beta}\right)^{\textrm{net}}}\right) \overline{w^{(p)}_{\beta b}}  \overline{w^{(p+1)}_{\alpha \eta}}\\
&\hspace{10mm}+g''_{p+1}\left(\overline{\left( x^{(p+1)}_{t\eta}\right)^{\textrm{net}}}\right) \overline{w^{(p)}_{\eta b}} \sum_{\alpha=1}^{K_{p+2}} E^{(p+2)}_{t\alpha} \overline{w^{(p+1)}_{\alpha \eta}}\\
&=\sum_{\beta=1}^{K_{p+1}} \left[ \sum_{\alpha=1}^{K_{p+2}} \frac{\partial E^{(p+2)}_{t\alpha}}{\partial \overline{x^{(p+1)}_{t\beta}}}\overline{w^{(p+1)}_{\alpha \eta}} \right]\\
&\hspace{20mm} \cdot g'_{p+1}\left(\overline{\left( x^{(p+1)}_{t\eta}\right)^{\textrm{net}}}\right) g'_{p+1}\left(\overline{\left( x^{(p+1)}_{t\beta}\right)^{\textrm{net}}}\right) \overline{w^{(p)}_{\beta b}}\\
&\hspace{10mm}+\left[\sum_{\alpha=1}^{K_{p+2}} E^{(p+2)}_{t\alpha} \overline{w^{(p+1)}_{\alpha \eta}} \right]g''_{p+1}\left(\overline{\left( x^{(p+1)}_{t\eta}\right)^{\textrm{net}}}\right) \overline{w^{(p)}_{\eta b}}.
\end{split}
\label{eq:deltabyxconj3}
\end{equation*}

\section{Convergence of the One-Step Newton Steplength Algorithm of Real-Valued Complex Functions}

\par
Let $f:\Omega\subseteq\C^k \to \R$, and consider a general minimization algorithm with sequence of iterates $\{ \z(n)\}$ given recursively by 
\begin{equation}
\z(n+1)=\z(n)-\mu(n)\p(n), \textrm{ } n=0,1,...,
\label{eq:GeneralIterate}
\end{equation}
where $\p(n)\in\C^k$ such that $-\p(n)$ is the direction from the $n$th iterate to the $(n+1)$th iterate and $\mu(n)\in \R$ is the learning rate or steplength which we allow to vary with each step.  We are interested in guaranteeing that the minimization algorithm is a descent method, that is, that at each stage of the iteration the inequality $f(\z(n+1))\le f(\z(n))$ for $n=0,1,...$ holds.   Here, we provide details of the proof of the one-step Newton steplength algorithm for the minimization of real-valued functions on complex domains.  Our treatment follows the exposition in \cite{Ortega1970}, with the application to the complex Newton algorithm providing a proof of Theorem \ref{thm:ConvergenceNewton}.

\par

\begin{lemma}
Suppose that $f:\Omega\subseteq\C^k \to \R$ is $\R$-differentiable at $\z\in\mathrm{int} (\Omega)$ and that there exists $\p\in\C^k$ such that $\mathrm{Re} \left( \frac{\partial f}{\partial \z}(\z) \p\right) >0$.  Then there exists a $\delta>0$ such that $f(\z-\mu \p)<f(\z) \textrm{ for all } \mu\in (0,\delta)$.
\label{lem:ExistenceofMu}
\end{lemma}

\begin{proof}
Let $\z=\x+i\y\in \textrm{int}(\Omega)$ with $\x,\y\in\R^k$.  The function $f:\Omega\subseteq\C^k\to\R$ is $\R$-differentiable at $\z$ if and only if $f:D\subseteq\R^{2k}\to\R$ is (Frechet) differentiable at $(\x,\y)^T\in\mathrm{int}(D)$, where $D$ is defined as in (\ref{eq:Ddef}) and the (Frechet) derivative (equal to the Gateau derivative) at $(\x,\y)^T$ is given by $\left( \frac{\partial f}{\partial \x}, \frac{\partial f}{\partial \y} \right)$.
Suppose there exists $\p=\p_R+i\p_I\in\C^k$ with $\p_R,\p_I\in \R^k$ such that  $\mathrm{Re} \left( \frac{\partial f}{\partial \z}(\z) \p\right) >0$.  Then using the coordinate and cogradient transformations (\ref{eq:CoordTransform}) and (\ref{eq:CogradTransform}) and the fact that $f$ is real-valued, we have the following (\cite{Kreutz-Delgado2009}, pg. 34):
\begin{equation}
\begin{split}
&\left( \frac{\partial f}{\partial \x}(\x,\y), \frac{\partial f}{\partial \y}(\x,\y)\right)\left( \begin{array}{c} \p_R \\ \p_I \end{array}\right) = \left( \frac{\partial f}{\partial \z}(\z,\overline{\z}), \frac{\partial f}{\partial \overline{\z}}(\z,\overline{\z})\right) J \cdot \frac{1}{2} J^* \left( \begin{array}{c} \p \\ \overline{\p} \end{array}\right)\\
& = \frac{\partial f}{\partial \z}(\z,\overline{\z}) \p +\frac{\partial f}{\partial \overline{\z}}(\z,\overline{\z})\overline{\p}  = \frac{\partial f}{\partial \z}(\z) \p +\overline{\frac{\partial f}{\partial \z}(\z)\p} = 2\mathrm{Re} \left( \frac{\partial f}{\partial \z}(\z) \p\right) >0.
\end{split}
\label{eq:ConditionTransform}
\end{equation}
By (8.2.1) in \cite{Ortega1970} there exists a $\delta>0$ such that $f((\x,\y)-\mu(\p_R,\p_I))<f(\x,\y) \textrm{ for all } \mu\in (0,\delta).$  Viewing $f$ again as a function on the complex domain $\Omega$, this is equivalent to the statement that $f(\z-\mu \p)<f(\z) \textrm{ for all } \mu\in (0,\delta)$.\end{proof}

\par
Recall from Section VI that a stationary point of $f$ to be a stationary point in the sense of the function $f(\z)=f(\x,\y):D\subseteq\R^{2k}\to\R.$  If $\hat{\z}=\hat{\x}+i\hat{\y}$ with $\hat{\x},\hat{\y}\in\R^k$, then $\hat{\z}$ is a stationary point of $f$ if and only if $\frac{\partial f}{\partial \x}(\hat{\x},\hat{\y})=\frac{\partial f}{\partial \y}(\hat{\x},\hat{\y})=0$.  Note that if $\Re \left( \frac{\partial f}{\partial \z}(\z)\right)\neq 0$ for $\z\in\textrm{int} (\Omega)$ (i.e. $\z$ is not a stationary point), then there always exists a $\p\in\C^k$ such that $\Re \left( \frac{\partial f}{\partial \z}(\z) \p\right)>0$.  So this result is always true in the real domain, and the proof of Lemma \ref{lem:ExistenceofMu} only translates the result from the real domain to the complex domain.

\par
For the sequence of iterates $\{ \z(n)\}$ given by (\ref{eq:GeneralIterate}), we can find a sequence $\{ \p(n)\}$ such that $\Re\left( \frac{\partial f}{\partial \z}(\z(n))\p(n)\right)>0$ for $n=0,1,...$.  By Lemma \ref{lem:ExistenceofMu}, for each $n$ there is at least one $\mu(n) \in (0,\infty)$ such that $f(\z(n)-\mu(n) \p(n))<f(\z(n))$.  At each step in the algorithm we would like to make the largest descent in the value of $f$ as possible, so finding a desirable steplength $\mu(n)$ to guarantee descent translates into the real one-dimensional  problem of minimizing $f(\z(n)-\mu\p(n))$ as a function of $\mu$.  For each $n$ let $\z(n)=\x(n)+i\y(n)$ and $\p(n)=\p_R(n)+i\p_I(n)$ with $\x(n),\y(n),\p_R(n),\p_I(n)\in\R^k$ and write
\begin{equation*}
f(\z(n)-\mu\p(n))=f((\x(n),\y(n))-\mu(\p_R(n),\p_I(n))).
\label{eq:RealMinProblem}
\end{equation*}
Suppose $f$ is twice $\R$-differentiable on $\Omega$.  As an approximate solution to this one-dimensional minimization problem we take $\mu(n)$ to be the minimizer of the second-degree Taylor polynomial (in $\mu$)
\begin{equation}
\begin{split}
T_2(\mu) &= f(\x(n),\y(n))\\
&\hspace{10mm}-\mu \left( \frac{\partial f}{\partial \x}(\x(n),\y(n)),\frac{\partial f}{\partial \y} (\x(n),\y(n))\right)\left( \begin{array}{c}\p_R(n) \\ \p_I(n) \end{array}\right)\\
&\hspace{10mm}+ \frac{1}{2} \mu^2 \left( \begin{array}{c}\p_R(n) \\ \p_I(n) \end{array}\right)^T\mathcal{H}_{\mathbf{r}\mathbf{r}}(\x(n),\y(n))  \left( \begin{array}{c}\p_R(n) \\ \p_I(n) \end{array}\right)
\end{split}
\label{eq:2ndTaylorExp}
\end{equation}
where $\mathcal{H}_{\mathbf{r}\mathbf{r}}$ denotes the real Hessian matrix
\begin{equation*}
\mathcal{H}_{\mathbf{r}\mathbf{r}} = \left( \frac{\partial}{\partial \x},\frac{\partial}{\partial \y}\right) \left( \frac{\partial f}{\partial \x}, \frac{\partial f}{\partial \y}\right)^T.
\label{eq:RealHess}
\end{equation*}
If
\begin{equation*}
\left(\begin{array}{c}\p_R(n) \\ \p_I(n) \end{array}\right)^T \mathcal{H}_{\mathbf{r}\mathbf{r}}(\x(n),\y(n))  \left( \begin{array}{c}\p_R(n) \\ \p_I(n) \end{array}\right) >0
\label{eq:MinCondition}
\end{equation*}
then $T_2$ has a minimum at
\begin{equation}
\mu(n)=\frac{\left( \frac{\partial f}{\partial \x}(\x(n),\y(n)),\frac{\partial f}{\partial \y} (\x(n),\y(n))\right)\left( \begin{array}{c}\p_R(n) \\ \p_I(n) \end{array}\right)}{\left(\begin{array}{c}\p_R(n) \\ \p_I(n) \end{array}\right)^T \mathcal{H}_{\mathbf{r}\mathbf{r}}(\x(n),\y(n))  \left( \begin{array}{c}\p_R(n) \\ \p_I(n) \end{array}\right)}.
\label{eq:ApproxSolnReal}
\end{equation}
(Note this is equivalent to taking one step toward minimizing $f$ over $\mu$ via the real Newton algorithm.)  Using a computation similar to (\ref{eq:ConditionTransform}) in the proof of Lemma \ref{lem:ExistenceofMu}, the denominator of (\ref{eq:ApproxSolnReal}) translates back into complex coordinates as (\cite{Kreutz-Delgado2009}, pg. 38):
\begin{equation}
\begin{split}
&\left(\begin{array}{c}\p_R(n) \\ \p_I(n) \end{array}\right)^T \mathcal{H}_{\mathbf{r}\mathbf{r}}(\x(n),\y(n))  \left( \begin{array}{c}\p_R(n) \\ \p_I(n) \end{array}\right)\\
&\hspace{10mm}=2 \Re \left\{\p(n)^* \mathcal{H}_{\z\z}(\z(n))\p(n) + \p(n)^*\mathcal{H}_{\overline{\z}\z}(\z(n))\overline{\p(n)} \right\}.
\end{split}
\label{eq:BilinearFormEq}
\end{equation}
Combining (\ref{eq:ApproxSolnReal}) with (\ref{eq:BilinearFormEq}) and (\ref{eq:ConditionTransform}), if 
\begin{equation}
\Re \left\{\p(n)^* \mathcal{H}_{\z\z}(\z(n))\p(n) + \p(n)^*\mathcal{H}_{\overline{\z}\z}(\z(n))\overline{\p(n)} \right\}>0
\label{eq:MinCondition}
\end{equation}
we can take the approximate solution to the minimization problem to be
\begin{equation}
\mu(n)=\frac{\mathrm{Re} \left\{ \frac{\partial f}{\partial \z}(\z(n)) \p(n)\right\}}{\Re \left\{\p(n)^* \mathcal{H}_{\z\z}(\z(n))\p(n) + \p(n)^*\mathcal{H}_{\overline{\z}\z}(\z(n))\overline{\p(n)} \right\}}.
\label{eq:ApproxSteplength}
\end{equation}
Notice that (\ref{eq:MinCondition}) is in fact both a necessary and sufficient condition to obtain an approximate solution using (\ref{eq:2ndTaylorExp}) to the one-dimensional minimization problem of $f(\z(n)-\mu\p(n))$ over $\mu$, for if
$$\Re \left\{\p(n)^* \mathcal{H}_{\z\z}(\z(n))\p(n) + \p(n)^*\mathcal{H}_{\overline{\z}\z}(\z(n))\overline{\p(n)} \right\}<0,$$
the Taylor polynomial (\ref{eq:2ndTaylorExp}) attains only a maximum.

\par
Since defining the sequence of steplengths $\{ \mu(n)\}$ by (\ref{eq:ApproxSteplength}) is only an approximate method, to guarantee the descent of the iteration, we consider further modification of the steplengths.  From Lemma \ref{lem:ExistenceofMu}, it is clear that we can choose a sequence of underrelaxation factors $\{ \omega(n)\}$ such that 
$$f(\z(n)-\omega(n)\mu(n)\p(n))<f(\z(n))$$
which guarantees that the iteration
\begin{equation}
\z(n+1)=\z(n)-\omega(n)\mu(n)\p(n), \textrm{ } n=0,1,...
\label{eq:DampedGeneralIterate}
\end{equation}
is a descent method.  We describe a way to choose the sequence $\{ \omega(n)\}$.

\par
First, recall some notation from Section VI.  Suppose $\Omega$ is open and let $\z(0)\in\Omega$. The level set of $\z(0)$ under $f$ on $\Omega$ is defined by (\ref{eq:levelset}), and $L_{\C^k}^0(f(\z(0)))$ is the path-connected component of $L_{\C^k}(f(\z(0))$ containing $\z(0)$.  Let $\Vert \cdot \Vert_{\C^k}:\C^k\to\R$ denote the Euclidean norm on $\C^k$, with $\Vert \z \Vert_{\C^k} =\sqrt{\z^* \z}$.

\begin{lemma}[Complex Version of the One-Step Newton Steplength Algorithm]
Let $f:\Omega\subseteq\C^k\to\R$ be twice-continuously $\R$-differentiable on the open set $\Omega$.  Suppose $L_{\C^k}^0(f(\z(0)))$ is compact for $\z(0)\in\Omega$ and that 
\begin{equation}
\eta_0 \mathbf{h}^*\mathbf{h} \le \mathrm{Re}\{\mathbf{h}^* \mathcal{H}_{\z\z}(\z)\mathbf{h} + \mathbf{h}^* \mathcal{H}_{\overline{\z}\z}(\z)\overline{\mathbf{h}}\} \le \eta_1\mathbf{h}^*\mathbf{h}
\label{eq:EtaCondition}
\end{equation}
for all $\z\in L_{\C^k}^0(f(\z(0)))$ and $\mathbf{h}\in\C^k$, where $0<\eta_0\le \eta_1$.  Fix $\epsilon \in (0,1]$.  Define the sequence $\{ \z(n)\}$ using (\ref{eq:DampedGeneralIterate}) with $\p(n)\neq0$ satisfying
\begin{equation}
\mathrm{Re}\left( \frac{\partial f}{\partial \z} (\z(n))(\p(n))\right) \ge 0,
\label{eq:ExistenceLemmaCondition}
\end{equation}\normalsize
$\mu(n)$ defined by (\ref{eq:ApproxSteplength}), and
\begin{equation}
0<\epsilon\le \omega(n) \le \frac{2}{\gamma(n)}-\epsilon,
\label{eq:UnderrelaxFactors}
\end{equation}
where, setting $\z=\z(n)$ and $\p=\p(n)$,
\begin{equation}
\begin{split}
\gamma(n)=&\sup \left. \left\{  \frac{\mathrm{Re}\{\p^* \mathcal{H}_{\z\z}(\z-\mu\p)\p + \p^* \mathcal{H}_{\overline{\z}\z}(\z-\mu\p)\overline{\p}\}}{\mathrm{Re}\{\p^* \mathcal{H}_{\z\z}(\z)\p + \p^* \mathcal{H}_{\overline{\z}\z}(\z)\overline{\p}\}} \right. \right\vert \\
&\hspace{45mm} \left.\begin{array}{c} \mu>0, \textrm{ } f(\z-\nu\p)<f(\z)\\
\textrm{for all }\nu\in(0,\mu]\end{array}\right\}.
\end{split}
\label{eq:GammaDef}
\end{equation}
Then $\{ \z(n)\}\subseteq L_{\C^k}^0(f(\z(0)))$,
\begin{equation*}
\lim_{n\to\infty} \frac{\mathrm{Re}\left( \frac{\partial f}{\partial \z} (\z(n))(\p(n))\right)}{\Vert \p(n)\Vert_{\C^k}}=0,
\label{eq:Limitto0Condition}
\end{equation*}
and $\lim_{n\to\infty} (\z(n)-\z(n+1))=0$.
\label{lem:SteplengthAlg}
\end{lemma}

\begin{proof}
Let $f:\Omega\subseteq\C^k\to\R$ be twice-continuously $\R$-differentiable on the open set $\Omega$, and define $D$ as in (\ref{eq:Ddef}). Then $D$ is open and $f(\x,\y):D\subseteq\R^{2k}\to\R$ is twice-continuously differentiable on $D$.  Let $\z(0)=\x(0)+i\y(0)\in\Omega$ with $\x(0),\y(0)\in\R^k$ and set 
$$L^0_{\R^{2k}}(f(\x(0),\y(0))=\left\{\left.\left(\begin{array}{c}\x \\ \y \end{array}\right)\in D \, \right\vert\, \begin{array}{c} \x,\y\in\R^k, \\ \z=\x+i\y\in L^0_{\C^{k}}(f(\z(0))) \end{array}\right\}.$$
It is clear that since $L_{\C^k}^0(f(\z(0)))$ is assumed to be compact, the real level set $L^0_{\R^{2k}}(f(\x(0),\y(0))$ is also compact.  
\par
Next, observe that for $\z=\x+i\y\in\C^k$ with $\x,\y\in\R^k$, if $\Vert \cdot\Vert_{R^{2k}}:\R^{2k}\to\R$ denotes the Euclidean norm on $\R^{2k}$, then
$$\Vert \z \Vert_{\C^k}^2=\z^*\z=\left\Vert \left( \begin{array}{c} \x \\ \y \end{array} \right) \right\Vert_{\R^{2k}}^2.$$
Using this fact and (\ref{eq:BilinearFormEq}) we see that for $\z=\x+i\y\in L_{\C^k}^0(f(\z(0)))$ and $\mathbf{h}=\mathbf{h}_R+i\mathbf{h}_I\in\C^k$ with $\x,\y,\mathbf{h}_R,\mathbf{h}_I\in\R^k$ the condition (\ref{eq:EtaCondition}) is equivalent to
\begin{equation*}
\eta_0' \left\Vert \left(\begin{array}{c} \mathbf{h}_R \\ \mathbf{h}_I\end{array} \right)\right\Vert_{\R^{2k}}^2 \le \left(\begin{array}{c} \mathbf{h}_R \\ \mathbf{h}_I\end{array} \right)^T \mathcal{H}_{\mathbf{r}\mathbf{r}}(\x,\y) \left(\begin{array}{c} \mathbf{h}_R \\ \mathbf{h}_I\end{array} \right) \le \eta_1' \left\Vert \left(\begin{array}{c} \mathbf{h}_R \\ \mathbf{h}_I\end{array} \right)\right\Vert_{\R^{2k}}^2,
\label{eq:RealEtaCondition}
\end{equation*}
where again $\mathcal{H}_{\mathbf{r}\mathbf{r}}$ denotes the real Hessian matrix of $f(\x,\y):D\subseteq\R^{2k}\to\R$, and $0<\eta_0'=\frac{\eta_0}{2}\le\frac{\eta_1}{2}=\eta_1'$.
\par
We have already seen in the proof of Lemma \ref{lem:ExistenceofMu} (see the calculation (\ref{eq:ConditionTransform})) that the condition (\ref{eq:ExistenceLemmaCondition}) on the vectors $\p(n)=\p_R(n)+i\p_I(n)$ with $\p_R(n),\p_I(n)\in\R^k$ is equivalent to the real condition
\begin{equation*}
\left( \frac{\partial f}{\partial \x}(\x(n),\y(n)),\frac{\partial f}{\partial \y}(\x(n),\y(n))\right) \left( \begin{array}{c}\p_R(n)\\ \p_I(n) \end{array}\right) \ge 0.
\label{eq:RealExistenceLemma}
\end{equation*}
We have also seen that our choice (\ref{eq:ApproxSteplength}) for $\mu(n)$ is equal to (\ref{eq:ApproxSolnReal}).
\par
Finally, for $\epsilon\in (0,1]$, using (\ref{eq:BilinearFormEq}) again we have the real analogue of (\ref{eq:GammaDef}):
\begin{equation*}
\begin{split}
&\gamma(n)=\\
&\sup \left. \left\{  \frac{\left(\begin{array}{c}\p_R(n)\\ \p_I(n)\end{array}\right)^T \mathcal{H}_{\mathbf{r}\mathbf{r}}((\x(n),\y(n))-\mu(\p_R(n),\p_I(n)))\left(\begin{array}{c}\p_R(n)\\ \p_I(n)\end{array}\right)}{\left(\begin{array}{c}\p_R(n)\\ \p_I(n)\end{array}\right)^T \mathcal{H}_{\mathbf{r}\mathbf{r}}(\x(n),\y(n))\left(\begin{array}{c}\p_R(n)\\ \p_I(n)\end{array}\right)} \right. \right\vert \\
&\hspace{10mm} \left.\begin{array}{c} \mu>0, \textrm{ } f((\x(n),\y(n))-\nu(\p_R(n),\p_I(n)))<f(\x(n),\y(n))\\
\textrm{for all }\nu\in(0,\mu]\end{array}\right\}.
\end{split}
\label{eq:RealGamma}
\end{equation*}
By (\ref{eq:ConditionTransform}),
\begin{equation*}
\frac{\left( \frac{\partial f}{\partial \x}(\x(n),\y(n)),\frac{\partial f}{\partial \y}(\x(n),\y(n))\right) \left(\begin{array}{c}\p_R(n)\\ \p_I(n)\end{array}\right)}{\left\Vert \left(\begin{array}{c}\p_R(n)\\ \p_I(n)\end{array}\right)\right\Vert_{\R^{2k}}}=\frac{2\Re \left(\frac{\partial f}{\partial \z}(\z(n)) \p(n) \right)}{\Vert \p(n)\Vert_{\C^k}},
\end{equation*}
so applying (14.2.9) in \cite{Ortega1970}, $\left\{ (\x(n), \y(n))^T\right\} \subseteq L^0_{\R^{2k}}(f(\x(0),\y(0)))$,
$$\lim_{n\to\infty}\frac{\left( \frac{\partial f}{\partial \x}(\x(n),\y(n)),\frac{\partial f}{\partial \y}(\x(n),\y(n))\right) \left(\begin{array}{c}\p_R(n)\\ \p_I(n)\end{array}\right)}{\left\Vert \left(\begin{array}{c}\p_R(n)\\ \p_I(n)\end{array}\right)\right\Vert_{\R^{2k}}}=0,$$
and 
$$\lim_{n\to\infty} \left( \left( \begin{array}{c} \x(n)\\ \y(n)\end{array}\right)-\left( \begin{array}{c} \x(n+1)\\ \y(n+1)\end{array}\right)\right)=0.$$
Translating back to complex coordinates yields the desired conclusion.\end{proof}

\par
Assume that there is a unique stationary point $ \hat{\z}$ in $L^0_{\C^k}(f(\z(0))$.  We desire to guarantee that the sequence of iterates $\{ \z(n)\}$ converges to $\hat{\z}$.  Before we give conditions for convergence of the complex version of the one-step Newton steplength algorithm, recall from Section VI that the R-factors of a sequence $\{\z(n)\}\subseteq\C^k$ that converges to $\hat{\z}\in\C^k$ are given by (\ref{eq:Rfactors}), and the sequence has at least an R-linear rate of convergence if $R_1\{\z(n)\}<1$.

\begin{lemma}[Convergence of the Complex Version of the One-Step Newton Steplength Algorithm]
Let $f:\Omega\subseteq\C^k\to\R$ be twice-continuously $\R$-differentiable on the open convex set $\Omega$ and assume that $L^0_{\C^k}(f(\z(0))$ is compact for $\z(0)\in\Omega$.  Assume the notation as in Lemma \ref{lem:SteplengthAlg}.  Suppose for all $z\in\Omega$,
\begin{equation}
\mathrm{Re}\{ \mathbf{h}^* \mathcal{H}_{\z\z}(\z)\mathbf{h}+\mathbf{h}^* \mathcal{H}_{\overline{\z}\z}(\z)\overline{\mathbf{h}}\}>0 \textrm{ for all } \mathbf{h}\in\C^k,
\label{eq:PosDefEquivCondition}
\end{equation}
and assume that the $\p(n)$ are nonzero vectors satisfying
\begin{equation}
\mathrm{Re}\left( \frac{\partial f}{\partial \z}(\z(n))\p(n)\right)\ge C \left\Vert \left( \frac{\partial f}{\partial \z}(\z(n))\right)^T\right\Vert_{\C^k} \Vert \p(n)\Vert_{\C^k}, \textrm{ } n=0,1,...
\label{eq:pCondition}
\end{equation}
for some fixed $C>0$. Assume $f$ has a unique stationary point $\hat{\z}$ in $L^0_{\C^k}(f(\z(0))$.  Then $\lim_{n\to\infty} \z(n)= \hat{\z}$, and the rate of convergence is at least R-linear.
\label{lem:GeneralConvergence}
\end{lemma}

\begin{proof}
As in the proof of Lemma \ref{lem:SteplengthAlg}, given the assumptions of this lemma, $f:D\subseteq\R^{2k}\to\R$  is twice-continuously (Frechet) differentiable on the open convex set $D$, and the set $L_{\R^{2k}}^0(f(\x(0),\y(0))$ is compact for $\z(0)=\x(0)+i\y(0)\in\Omega$, where $\x(0),\y(0)\in\R^k$.

\par
Using (\ref{eq:BilinearFormEq}), for $\z=\x+i\y\in\Omega$ the condition (\ref{eq:PosDefEquivCondition}) is equivalent to the condition 
\begin{equation*}
\left( \begin{array}{c} \mathbf{h}_1 \\ \mathbf{h}_2 \end{array}\right)^T \mathcal{H}_{\mathbf{r}\mathbf{r}}(\x,\y)\left( \begin{array}{c} \mathbf{h}_1 \\ \mathbf{h}_2 \end{array}\right)>0 \textrm{ for all } \left( \begin{array}{c} \mathbf{h}_1 \\ \mathbf{h}_2 \end{array}\right)\in\R^{2k} \textrm{ with }\mathbf{h}_1,\mathbf{h}_2\in\R^k.
\end{equation*}
Thus for all $(\x,\y)^T\in D$, the real Hessian $\mathcal{H}_{\mathbf{r}\mathbf{r}}(\x,\y)$ of $f$ is positive definite.

\par
Also as in the proof of Lemma \ref{lem:SteplengthAlg}, the real versions of the definitions of $\mu(n)$ and $\omega(n)$ given by (\ref{eq:ApproxSteplength}) and (\ref{eq:UnderrelaxFactors}), respectively, satisfy the real one-step Newton steplength algorithm (14.2.9) in \cite{Ortega1970}.

\par
Finally, for $\z=\x+i\y\in\C^k$ with $\x,\y\in\R^k$, a simple calculation shows that
$$2\left\Vert\left( \frac{\partial f}{\partial \z} (\z)\right)^T \right\Vert_{\C^k} = \left\Vert \left( \frac{\partial f}{\partial \x}(\x,\y),\frac{\partial f}{\partial \y}(\x,\y)\right)^T\right\Vert_{\R^{2k}},$$
so using the calculation (\ref{eq:ConditionTransform}) in the proof of Lemma \ref{lem:ExistenceofMu}, the condition (\ref{eq:pCondition}) for the nonzero vectors $\p(n)=\p_R(n)+i\p_I(n)$ with $\p_R(n),\p_I(n)\in\R^k$ is equivalent to the real condition
\begin{equation*}
\begin{split}
&\left( \frac{\partial f}{\partial \x}(\x,\y),\frac{\partial f}{\partial \y}(\x,\y)\right) \left( \begin{array}{c} \p_R(n) \\ \p_I(n) \end{array}\right) \\
&\hspace{15mm}\ge C \left\Vert \left( \frac{\partial f}{\partial \x}(\x,\y),\frac{\partial f}{\partial \y}(\x,\y)\right)^T \right\Vert_{\R^{2k}}  \left\Vert \left( \begin{array}{c} \p_R(n) \\ \p_I(n) \end{array}\right) \right\Vert_{\R^{2k}}.
\end{split}
\end{equation*}
Thus we may apply Theorem (14.3.6) in \cite{Ortega1970} and transfer back to complex coordinates to obtain that $\lim_{n\to\infty}\z(n)= \hat{\z}$, where $ \hat{\z}= \hat{\x}+i \hat{\y}$ with $\hat{\x},\hat{\y}\in\R^k$ is the unique stationary point of $f$ in $L^0_{\C^k}(f(\z(0)))$, and the rate of convergence is at least R-linear. \end{proof}

\par
We now apply the previous results to the complex Newton algorithm.  Let $f:\Omega\subseteq\C^k\to\R$ be twice-continuously $\R$-differentiable on the open convex set $\Omega$.  Let $\z(0)\in\Omega$ and assume that the level set $L^0_{\C^k}(f(\z(0)))$ is compact.  Suppose for all $\z\in\Omega$,
\begin{equation*}
\Re\{ \mathbf{h}^* \mathcal{H}_{\z\z}(\z)\mathbf{h} + \mathbf{h}^* \mathcal{H}_{\overline{\z}\z}(\z)\overline{\mathbf{h}}\}>0 \textrm{ for all }\mathbf{h}\in\C^k.
\label{eq:PosDefEquiv2}
\end{equation*}
As in the proof of Lemma \ref{lem:GeneralConvergence}, this condition is equivalent to the positive definiteness of the real Hessian matrix $\mathcal{H}_{\mathbf{r}\mathbf{r}}(\x,\y)$ of $f$ for all $(\x,\y)^T\in D$.  Since $f$ is twice-continuously $\R$-differentiable, the Hessian operator $\mathcal{H}_{\mathbf{r}\mathbf{r}}(\cdot):D\subseteq\R^{2k}\to L(\R^{2k})$ (where $L(\R^{2k})$ denotes the set of linear operators $\R^{2k}\to\R^{2k}$) is continuous.  Restricting to the compact set $L^0_{\R^{2k}}(f(\x(0),\y(0)))$ we have that $\mathcal{H}_{\mathbf{r}\mathbf{r}}(\cdot):L^0_{\R^{2k}}(f(\x(0),\y(0)))\to L(\R^{2k})$ is a continuous mapping such that $\mathcal{H}_{\mathbf{r}\mathbf{r}}(\x,\y)$ is positive definite for each vector $(\x,\y)^T\in L^0_{\R^{2k}}(f(\x(0),\y(0)))$.  For each $(\x,\y)^T\in L^0_{\R^{2k}}(f(\x(0),\y(0)))$ set
\begin{equation*}
 \tilde{\p}(\x,\y)=\mathcal{H}_{\mathbf{r}\mathbf{r}}(\x,\y)^{-1}\left( \frac{\partial f}{\partial \x}(\x,\y),\frac{\partial f}{\partial \y}(\x,\y) \right)^T.
\label{eq:TildepDefReal}
\end{equation*}
By Lemma (14.4.1) in \cite{Ortega1970}, there exists a constant $C>0$ such that
\begin{equation}
\begin{split}
&\left( \frac{\partial f}{\partial \x}(\x,\y),\frac{\partial f}{\partial \y}(\x,\y) \right) \tilde{\p}(\x,\y)\\
&\hspace{20mm}\ge C \left\Vert \left( \frac{\partial f}{\partial \x}(\x,\y),\frac{\partial f}{\partial \y}(\x,\y) \right)^T \right\Vert_{\R^{2k}}\Vert  \tilde{\p}(\x,\y) \Vert_{\R^{2k}}
\end{split}
\label{eq:RealGradRelated}
\end{equation}
for all $(\x,\y)^T\in L^0_{\R^{2k}}(f(\x(0),\y(0)))$.  As in the proof of Lemma \ref{lem:GeneralConvergence}, (\ref{eq:RealGradRelated}) is equivalent to the inequality
\begin{equation*}
\mathrm{Re}\left( \frac{\partial f}{\partial \z}(\z)\p(\z)\right)\ge C \left\Vert \left( \frac{\partial f}{\partial \z}(\z)\right)^T\right\Vert_{\C^k} \Vert \p(\z)\Vert_{\C^k}
\label{eq:ComplexGradRelated}
\end{equation*}
for all $\z\in L^0_{\R^{2k}}(f(\x(0),\y(0)))$, where
\begin{equation*}
\begin{split}
\tilde{\p}(\z)&= -\left[\mathcal{H}_{\z\z}(\z)-\mathcal{H}_{\overline{\z}\z}(\z)\mathcal{H}_{\overline{\z}\overline{\z}}(\z)^{-1}\mathcal{H}_{\z\overline{\z}}(\z)\right]^{-1} \\
& \hspace{10mm} \cdot\left[ \mathcal{H}_{\overline{\z}\z}(\z)\mathcal{H}_{\overline{\z}\overline{\z}}(\z)^{-1} \left(\frac{\partial f}{\partial \overline{\z}}(\z) \right)^*-\left(\frac{\partial f}{\partial \z}(\z) \right)^*\right]
\end{split}
\label{eq:pTildeDefComplex}
\end{equation*}
is obtained from $\tilde{\p}(\x,\y)$ (where $\z=\x+i\y$) using the coordinate and cogradient transformations (\ref{eq:CoordTransform}) and (\ref{eq:CogradTransform}), respectively \cite{Kreutz-Delgado2009}.  Suppose $f$ has a unique stationary point $\hat{\z}$ in $L^0_{\C^k}(f(\z(0)))$.  Consider the iteration
\begin{equation*}
\z(n+1)=\z(n)-\omega(n)\mu(n)\p(n), \textrm{ } n=0,1,...,
\label{eq:DampedGeneralIterate2}
\end{equation*}
where the $\p(n)$ are the nonzero complex Newton updates defined by $\p(n)=\tilde{\p}(\z(n))$, and assume the notation of Lemma \ref{lem:SteplengthAlg}.  Then $\{ \z(n)\}\subseteq L^0_{\C^k}(f(\z(0)))$. The vectors $\p(n)$ satisfy (\ref{eq:pCondition}), so by Lemma \ref{lem:GeneralConvergence} the sequence of iterates $\{ \z(n)\}$ converges to $\hat{\z}$, and the rate of convergence is at least R-linear.  Thus we have proved Theorem \ref{thm:ConvergenceNewton}.



%

\end{document}